\documentclass[10pt, 
]{amsart}
\usepackage{amssymb,mathrsfs}
\usepackage{amssymb}
\usepackage{amsmath}
\usepackage{amsfonts, mathrsfs}
\usepackage{epsfig}
\usepackage{color}
\usepackage{epstopdf}
\usepackage{graphicx}
\usepackage{srcltx}
\usepackage{amsthm}
\usepackage{enumerate}
\usepackage[mathscr]{eucal}
\usepackage{verbatim}
\usepackage[draft]{todonotes}
\usepackage{relsize}
\usepackage[
]
{hyperref}

\usepackage[normalem]{ulem}

\hypersetup{
	colorlinks=false,       
	linkcolor=blue,          
	citecolor=red,        
	filecolor=magenta,      
	urlcolor=cyan           
}

\newcommand{\Be}{\begin{equation}}
\newcommand{\Ee}{\end{equation}}
\newcommand{\Bea}{\begin{eqnarray}}
\newcommand{\Eea}{\end{eqnarray}}
\newcommand{\Bel}{\begin{align}}
\newcommand{\Eel}{\end{align}}
\newcommand{\Beas}{\begin{eqnarray*}}
	\newcommand{\Eeas}{\end{eqnarray*}}
\newcommand{\Benu}{\begin{enumerate}}
	\newcommand{\Eenu}{\end{enumerate}}
\newcommand{\Bi}{\begin{itemize}}
	\newcommand{\Ei}{\end{itemize}}

\numberwithin{equation}{section}
\newcommand{\supp} {\text{supp\! }}
\newcommand{\dist} {\text{dist\! }}

\theoremstyle{plain}
\newtheorem{thm}{Theorem}[section]

\newtheorem{lem}[thm]{Lemma}
\newtheorem{prop}[thm]{Proposition}

\theoremstyle{remark}
\newtheorem{rmk}{Remark}

\theoremstyle{definition}
\newtheorem{defn}{Definition}[section]

\newcommand{\Angle}{\text{Angle}}
\DeclareMathOperator{\interior}{int}

\begin{document}

\title[Convergence of Fractional Schr\"odinger Equation]
{Pointwise convergence of the \\ fractional Schr\"odinger equation in $\mathbb R^2$}

\author{Chu-hee Cho}
\author{Hyerim Ko}

\address{Department of Mathematical Sciences and RIM, Seoul National University, Seoul 151--747, Republic of Korea}
\email{akilus@snu.ac.kr} 
\email{kohr@snu.ac.kr}

\thanks{This work was supported by 
NRF grant no. 2021R1A2B5B02001786,
 2020R1I1A1A01072942 (C. Cho)
 2017R1D1A1A02019547, 
and 2019R1A6A3A01092525 (H. Ko)
(Republic of Korea).}


\begin{abstract} 
We investigate the pointwise convergence of the solution to the fractional Schr\"odinger equation in $\mathbb R^2$. By establishing $H^s(\mathbb R^2)-L^3(\mathbb R^2)$ estimates for the associated maximal operator
provided that $s>1/3$, we improve the previous result obtained by Miao, Yang, and Zheng \cite{MYZ}.
Our estimates extend the refined Strichartz estimates obtained by Du, Guth, and Li \cite{DGL} to a general class of elliptic functions. 
\end{abstract}

\maketitle

\section{Introduction}
For $\alpha>1$, we consider the fractional Schr\"odinger equation 
\begin{align}\label{fraceq}
i\partial_t u+(-\Delta)^{\alpha/2}=0, \qquad u(x,0)=f(x)
\end{align}
for $f \in H^s(\mathbb R^2)$. Here, $H^s$ is the $L^2$ Sobolev space of order $s$. Formally, the solution of \eqref{fraceq} can be written as
\begin{align*}
U_\alpha f(x,t)= (2\pi)^{-2}\int_{\mathbb R^2} 
e^{i(x\cdot \xi + t|\xi|^\alpha)} \widehat{f}(\xi) d\xi.
\end{align*}
In this study, we investigate the order of  $s$ for which
\begin{equation}\label{convs}
\lim_{t \rightarrow 0}
U_\alpha f (x,t) = f(x) \quad  \text{a.e.}\,\, x
\end{equation}
holds whenever $f \in H^s(\mathbb R^2)$.

The problem of determining the optimal regularity $s$ for which \eqref{convs} holds for the Schr\"odinger equation was initially studied by Carleson \cite{C}. When $d=1$, he proved the convergence of \eqref{convs} with $\alpha=2$ for $s\ge1/4$, whereas it generally fails for $s<1/4$ in any dimension, as shown by Dahlberg and Kenig \cite{DK}.

In higher dimensions, Sj\"olin \cite{S} and Vega \cite{V} independently showed that \eqref{convs} with $\alpha=2$ holds for $s> 1/2$. This result was improved to $s>1/2-1/(4d)$ by Lee \cite{L} for $d=2$ and  by Bourgain \cite{B} for $d\ge 3$.
Subsequently, Bourgain \cite{B3} showed that $s \ge d /(2d+2)$ is necessary for the almost everywhere convergence. The sufficiency part of the convergence was shown by Du, Guth, and Li \cite{DGL} when $d=2$ and by Du and Zhang \cite{DZ} when $d\ge 3$ for a sharp range
except for the endpoint (see \cite{Carbery, Cowling, B0, MVV, TV, LR1, LR2, DGLZ} for previous work).

For the fractional Schr\"odinger operator ($\alpha>1$), Sj\"olin \cite{S} proved that \eqref{convs} holds if and only if $s \ge 1/4$ when $d=1$. He also obtained some positive results in higher dimensions: \eqref{convs} is valid for $s \ge 1/2$ when $d=2$ and for $s > 1/2$ when $d\ge3$. 
Subsequently, this result was improved by Miao, Yang, and Zheng \cite{MYZ} to $s> 3/8$ when $d=2$ and  $s>s_0$ for some $s_0<1/2$ when $d\ge3$. We extend the result for $d=2$.

\begin{thm}\label{point3}
Let $\alpha>1$. Then, \eqref{convs} holds for $ f \in H^s(\mathbb R^2)$ whenever $s>1/3$.
\end{thm}

The result in Theorem \ref{point3} extends to the solution of the linear dispersive equation  
\[
i u_t -\Phi(D) u=0, \qquad u(x,0) = f(x).
\]
Here, $\Phi(D)$ is a multiplier operator defined on $\mathbb R^2$, where $\Phi$ is a smooth function except for the origin and satisfies the following property: for $\alpha>1$, there is a constant $C\ge1$ such that $ |\nabla \Phi(\xi)| \ge C^{-1} |\xi|^{\alpha-1}$ and $|\partial_\xi^\gamma \Phi(\xi)| \le C |\xi|^{\alpha-|\gamma|}$ for any multi-indices $\gamma$.
See Remark \ref{disp}.

We denote by $B^d(x,r)$ a ball of radius $r$ centered at $x$ in $\mathbb R^d$. Theorem \ref{point3} follows from the maximal estimate.
\begin{thm}\label{point0}
Let $\alpha>1$. Then, for $s>1/3$, there exists a constant $C>0$ such that
\begin{equation*}\label{point5}
\big\|  U_\alpha f\big\|_{L_x^3L_t^\infty (B^2(0,1)\times [0,1])}
\le C \|f\|_{H^s(\mathbb R^2)}.
\end{equation*}
\end{thm}

The proof of Theorem \ref{point0} is motivated by the argument used in \cite{DGL} and proceeds by
using polynomial partitioning to decompose $U_\alpha f$ into cells as well as transversal and tangential parts of a wall. The first two parts are easy to handle by induction, whereas the tangential term is much more complicated. 
To treat the tangential part, we need to prove refined Strichartz estimate  for $U_\alpha$.
We prove the estimate by using the decoupling inequality for elliptic parabola and induction on scales via rescaling. 
In contrast to the Schr\"odinger operator, $U_\alpha f$ $(\alpha \neq2)$ does not preserve the form after parabolic rescaling. To circumvent this issue, we consider a class of general elliptic functions  as in \cite{G}. 
Thus, we obtain the refined Strichartz estimates for a general class of operators (see Proposition 3.4).

\subsubsection*{Structure of the paper}
The remainder of this paper is organized as follows. By applying polynomial partitioning, we reduce the problem to a problem of proving bilinear tangential estimate (Theorem \ref{8.1}). Section \ref{sec3} establishes the linear refined Strichartz estimates (Proposition \ref{linear}) and bilinear refined Strichartz estimates (Proposition \ref{bilinear}). Accordingly, we prove Theorem \ref{8.1}.

\subsubsection*{Notation}
Throughout the paper,  $\mathcal F(f)$ denotes the Fourier transform of $f$. Further, $A \lesssim B$ denotes $A \le C B$ for some constant $C>0$ and $\#\mathcal D$ denotes the cardinality of a set $\mathcal D$.

\section{Proof of Theorem \ref{point0}}\label{sec2}
Let $\alpha>1$ and set $\mathbb A_r$ be the annulus given by
\[
\mathbb A_r :=\{ \xi \in \mathbb R^2 : 2^{-1}r \le |\xi | \le 2r \}.
\] 
To prove Theorem \ref{point0}, by the Littlewood--Paley decomposition and the triangle inequality, it suffices to show that for any $\epsilon>0$, there is $C_\epsilon>0$ such that
\begin{equation*}\label{point6}
\big\| \sup_{0 <t \le 1 } |U_\alpha f| \big\|_{L^3(B^2(0,1))}
\le C_\epsilon R^{\frac13+\epsilon} \|f\|_2,
\end{equation*}
provided that $\widehat f$ is supported on $\mathbb A_R$ for $R \ge 1$. By a parabolic rescaling $\xi \rightarrow R\xi$ and $(x,t) \rightarrow (R^{-1}x, R^{-\alpha} t)$, the estimate is reduced to showing that
\begin{equation}\label{point7}
\big\| \sup_{0 <t \le R^\alpha } |U_\alpha f|\big\|_{L^3(B^2(0,R))}
\le C_\epsilon R^{\epsilon} \|f\|_2,
\end{equation}
whenever $\widehat f$ is supported on $\mathbb A_1$.
Now we reduce the matter to showing \eqref{point7} in which the supremum is taken over a smaller interval
$[0,R]$ instead of $[0,R^{\alpha}]$.
More precisely, to prove \eqref{point7} it suffices to show that
for any $\epsilon>0$, there exists a constant $C_\epsilon > 0$ such that
 \begin{equation*}
\| \sup_{0<t\le R} |U_\alpha f| \|_{L^3(B^2(0,R))} \le C_\epsilon R^{\epsilon} \|f\|_2
\end{equation*}
whenever $\widehat{f}$ is supported on $\mathbb A_1$. This reduction can be obtained by applying the time localization lemma in \cite[Lemma 2.11]{MYZ} for the fractional Schr\"odinger operator
(see also \cite{L}). Alternatively, one may verify the lemma by using $TT^*$ argument as in \cite[Lemma 2.1]{CLV}.
After finite decomposition, we may assume that $\widehat f$ is supported on a ball $B^2(\xi_0,r) \subset \mathbb A_1$. Hence, Theorem \ref{point0} is a consequence of the following.
For simplicity, let $B_R=B^2(0,R)\times [0,R]$.
\begin{thm}\label{main result}
Let $p\ge3$ and $R\ge1$. Then, for any $\epsilon > 0$, $q > \epsilon^{-4}$, and $r \le 1$ such that $B(\xi_0,r) \subset \mathbb A_1$,  there exists a constant $C_\epsilon > 0$ such that
\begin{equation}\label{main ineq}
\| U_\alpha f \|_{L_x^pL_t^q(B_R)}
\le C_\epsilon r^{\epsilon^2}R^{\epsilon}\|f\|_2
\end{equation}
whenever $\widehat{f}$ is supported on $B(\xi_0,r)$.
\end{thm}

Indeed, by the dominated convergence theorem, the estimate \eqref{main ineq} implies that
\[
\| U_\alpha f\|_{L_x^pL_t^\infty(B_R)} \le C_\epsilon R^\epsilon \|f\|_2
\]
for any $p>3$. By interpolating this with a trivial estimate $\| U_\alpha f \|_{L_x^2L_t^\infty(B_R)} \lesssim R^{1/2}\|f\|_2$, we have Theorem \ref{point0}.

We begin by stating a wave packet decomposition of $U_\alpha f$ (see, for example, \cite{Tao, DGL}).
For later use, we state the following for a more general operator $e^{it\Phi}f$ defined by
\[
e^{it\Phi}f(x,t) = \int e^{i(x\cdot \xi+t \Phi(\xi))} \widehat f(\xi)\,d\xi
\]
where $\Phi$ is smooth and the Hessian matrix of $\Phi$ is nondegenerate.
Let $\psi$ be a smooth function such that $\widehat \psi$ is supported on $B^2(0,3/2)$ and $\sum_{k\in \mathbb Z^2} |\widehat\psi(\cdot-k)|^2 =(2\pi)^{-2}$ on $\mathbb R^2$.
For $\delta>0$ and $(y,\nu) \in [R^{1/2}\mathbb Z^2 \cap B^2(0,R)] \times [R^{-1/2}\mathbb Z^2 \cap B^2(0,2)]$, 
we define a tube $T=T_{y,v}$  by
\begin{align}\label{tube}
T = \{ (x,t) \in \mathbb R^3 : |x-y+ t\nabla\Phi(v)| \le R^{1/2+\delta},\,0 \le t \le R\},
\end{align}
and denote the direction of tube by $D(T)=(-\nabla \Phi(v), 1)$ and the set of all tubes $T$ by $\mathcal T$. 
We define $\psi_T=\psi_{T_{y,v}}$ by  
\[\widehat{\psi_T}(\xi)=e^{-iy\cdot \xi}R^{1/2}\widehat\psi(R^{1/2}(\xi-v))\]
so that $\sum_{T \in \mathcal T} \psi_T(x) \overline {\mathcal F(\psi_T)(\xi)}=(2\pi)^{-2}e^{ix\cdot \xi}$ by the Poisson summation formula
(see for example \cite{Gr}).

\begin{lem} \label{wave packet}
Let $\Phi$, $T$ and $\psi_T$ be as above. Suppose $\widehat f$ is supported on the  ball $B^2(0,1)$. By setting
$f_T=\langle f, \psi_T \rangle \psi_T,$
we have
\begin{align*}\label{wavepacket}
f= \sum_{T \in \mathcal T} f_T
\end{align*}
such that
\[
\sum_{T \in \mathcal T} |\langle f, \psi_T \rangle|^2 \lesssim \|f\|_2^2
\]
and for sufficiently large $N\ge1$ and $(x,t) \in B^3(0,R)$,
\[
|e^{it\Phi} \psi_T(x,t)| \lesssim R^{-1/2} \chi_{T}(x,t) +O(R^{-N}) \|f\|_2.
\]
\end{lem}

Let $\epsilon > 0$ and $0 < r \le 1 \le R$. Suppose that the support of $\widehat f$ is contained in $B^2(\xi_0,r)$ for some $\xi_0 \in \mathbb A_1$. The proof proceeds by induction on the size of $r$ and $R$. 
Note that \eqref{main ineq} holds trivially for $R \sim 1$; hence, it suffices to consider  $R \gg 1$.
Furthermore, we only need to consider $r \ge R^{-1/2}$. In fact, if $r \le R^{-10}$, then \eqref{main ineq} follows trivially since $| U_\alpha f(x,t)| \le r\|f\|_2$ by H\"older's inequality. On the other hand, if $R^{-10} \le r \le R^{-\frac12}$, then all the wave packets have the same direction. Therefore, we apply H\"older's inequality and obtain
\begin{equation}
\begin{aligned}\label{one} 
\big\| \sum_{T} U_\alpha f_{T} \big\|_{L_x^pL_t^q(B_R)}
	&\lesssim
R^{-1/2+1/q} \| \sum_{T} \langle f ,\psi_T\rangle  \chi_{T} \|_{L_x^pL_t^\infty(B_R)}  \\ 
	&\lesssim
R^{-1/2+1/q} \sum_{T} | \langle f , \psi_T \rangle |  \| \sup_{T} \chi_{T} \|_{L_x^pL_t^\infty(B_R)}  \\ 
	&\lesssim
R^{-1/2+1/q} R^{(3/2+\delta)/p} \|f\|_2.
\end{aligned}
\end{equation}
For the last inequality, we use \eqref{tube}. Thus, \eqref{main ineq} follows for $p\ge 3$ and sufficiently large $q>\epsilon^{-4}$ with small $\delta=\delta(\epsilon)$.
Hereafter, we only consider $r\ge R^{-1/2}$.
Now, we may assume that $\eqref{main ineq}$ holds 
if the radius of balls in physical space is less than $R/2$ or the radius of balls in physical space is less than $R$ and that of balls in frequency space is less than $r/2$. Then, it suffices to show \eqref{main ineq} for $R \gg 1$ and
$r\ge R^{-1/2}$.

Now, we reduce the matter to showing the \text{\it bilinear tangential estimate} (Theorem \ref{8.1})
by a standard argument using polynomial partitioning. Let us denote by $Z(P)$ the zero set of a polynomial $P$. We say that $P$ is a nonsingular polynomial if $\nabla P(z) \neq 0$ for all $z \in Z(P)$. Throughout this paper, we may assume that the polynomial $P$ is a product of nonsingular polynomials by the density argument (see \cite{G}). We recall the polynomial partitioning in \cite{DGL}.

\begin{thm} \label{polynomial}
Let $g \in L_x^1L_t^s(\mathbb R^{d+1})$ be a nonzero function, $1 \le s < \infty$, and $D>0$. 
Then, there exists a nonzero polynomial $P$  defined on $\mathbb R^{d+1}$ of degree $\le D$, which is a product of distinct nonsingular polynomials, and there exists a collection of disjoint open sets $\{O_i\}_{i \in \mathcal I}$  such that $\# \mathcal I  \sim D^{d+1}$ and
\begin{equation*}\label{decomp}
(\mathbb R^d \times \mathbb R)\setminus Z(P) = \bigcup_{i \in \mathcal I}O_i.
\end{equation*}
Moreover, there exists a constant $C_1$ independent of $i$ such that
\[	
\|  g \|_{L_x^1L_t^s(\mathbb R^{d+1})} \le C_1 D^{d+1} \| \chi_{O_i} g \|_{L_x^1L_t^s(\mathbb R^{d+1})}
\]
for each $i \in \mathcal I$.
\end{thm}
By taking $s = q/p$, $D = R^{\epsilon^4}$, and $g=\chi_{B_R}  |U_\alpha f|^p$,
and applying Theorem \ref{polynomial}, we have
\begin{equation}\label{d3}
\|   U_\alpha f \|_{L_x^pL_t^q(B_R)}^p
	\le C_1 D^{3} \|  \raisebox{.3ex}{$\chi_{O_i}$}  U_\alpha f \|_{L_x^pL_t^q(B_R)}^p.
\end{equation}
We denote by $W$ a wall that is an $R^{\frac12 + \delta}$-neighborhood of  $Z(P)$
and a cell $\widetilde{O}_i = O_i \setminus W$. It is clear that
\begin{equation*}\label{Lp norm}
\|  U_\alpha f \|_{L_x^pL_t^q(B_R)}^p 
	= \sum_{i \in \mathcal I}\|  \raisebox{.5ex}{$\chi_{\widetilde{O}_i}$} U_\alpha f \|_{L_x^pL_t^q(B_R)}^p 
		+ \| \raisebox{.2ex}{$\chi_{W}$} U_\alpha f  \|_{L_x^pL_t^q(B_R)}^p.
\end{equation*}
If the terms $\sum_{i \in \mathcal I}\|  \raisebox{.5ex}{$\chi_{\widetilde{O}_i}$} U_\alpha f \|_{L_x^pL_t^q(B_R)}^p $ dominate the other in the above-mentioned equation, the estimate is easy to handle.
Hence, we consider this case first.

\noindent{\it Cellular part}.
Let us consider a subcollection $\widetilde {\mathcal I}$ of an index set $\mathcal I$ such that
\[
\widetilde{\mathcal I} = \big\{ i \in \mathcal I: \|   U_\alpha f \|_{L_x^pL_t^q(B_R)}^p 
	\le 2C_1 D^{3}  \| \raisebox{.5ex}{$\chi_{\widetilde{O}_i}$} U_\alpha f   \|_{L_x^pL_t^q(B_R)}^p \big\}
\]
where the constant $C_1$ is given by \eqref{d3}. 
To treat the case in which the cellular part dominates the walls, we may assume that 
 $\widetilde{\mathcal I}=\mathcal I$.
For each $i \in \mathcal I$, we set
\[
f_i = \sum_{T: T \cap \widetilde{O}_i \neq \emptyset} f_T.
\]
By Lemma \ref{wave packet}, 
if  $(x,t)\in \widetilde O_i$, then $ |U_\alpha f (x,t)| \le  |U_\alpha f_i(x,t)| + O(R^{-N})\|f\|_2 $ for sufficiently large $N$. 
Since each tube $T$ intersects at most $(D+1)$ cells $O_i$, we have
\[
\sum_{i \in \mathcal I}\|f_i\|_2^2 
	=\sum_{\substack{i \in \mathcal I, 
T \in \mathcal T : \\ T \cap \widetilde O_i \neq \emptyset}}\|f_T\|_{2}^2 \lesssim D\|f\|_{2}^2.
\]
Since $\# \mathcal I \sim D^3$, by pigeonholing, there exists an index $i_\circ \in \mathcal I$  such that $\|f_{i_\circ}\|_2^2 \lesssim D^{-2}\|f\|_{2}^2$.
We cover $B_R$ by $\{B_{R/2}'\}$, which are translations of $B_{R/2}$, and obtain
\begin{align*}
\| U_\alpha f \|_{L_x^pL_t^q(B_R)}^p
	&\le 2C_1D^3\sum_{B_{R/2}'}\| \sum_{T\,:\,T \cap \widetilde O_{i_\circ} \neq \emptyset} U_\alpha f_T \|_{L_x^pL_t^q(B_{R/2}')}^p 
	+ O(R^{-N})\|f\|_2^p.
\end{align*}
By applying the induction hypothesis, it follows that
\begin{align*} 
\| U_\alpha f \|_{L_x^pL_t^q(B_R)}^p 
	&\le CD^3 \big[C_\epsilon r^{\epsilon^2}(R/2)^{\epsilon}\|f_{i_\circ}\|_2 \big]^p + O(R^{-N})\|f\|_2^p \\ 
	&\le 2C2^{-p\epsilon} D^3 \big[C_\epsilon r^{\epsilon^2}R^{\epsilon} D^{-1}\|f\|_2 \big]^p. 
\end{align*}
Since $p>3$ and $D=R^{\epsilon^4}$, we see that $2C2^{-p\epsilon}D^{3-p} \le 2^{-p} $ for sufficiently large $R$.
Therefore, we get \eqref{main ineq} when the cell part dominates.

Now, we consider the opposite case in which the wall part dominates.
To this end, we present some definitions. We denote by $T_z(Z(P))$ the tangent plane of $Z(P)$ at a fixed point $z$. 
Let us partition $B_R$ into balls $B_j$ of radius $R^{1-\delta}$.
We say that a tube $T$ is {\it tangent} to the wall $W$ in $B_j$ if $T$ intersects $B_j$ and $W$ and satisfies
\[
\Angle \big( D(T),T_z(Z(P) \big) \le R^{-\frac12 + 2\delta}
\]
for any nonsingular point $z \in Z(P) \cap 10T \cap 2B_j$.
Otherwise, we say that $T$ is {\it transversal} to the wall $W$ in $B_j$. 
Let $\mathcal T_{j,\mathrm{tang}}$ be the collection of all tubes $T \in \mathcal T$ such that $T$ is tangent to the wall in $B_j$ and let $\mathcal T_{j,\mathrm{trans}}$ be the collection of tubes such that $T$ is transversal.
We also set
\[
f_{j,\mathrm{tang}} = \sum_{T \in \mathcal T_{j,\mathrm{tang}}}f_T
\quad \mbox{ and } \quad 
f_{j,\mathrm{trans}}= \sum_{T \in \mathcal T_{j,\mathrm{trans}}}f_T.
\]

For a given $\delta'>0$, we say that a tube $T$ is $R^{-1/2+\delta'}$-tangent to $Z$ if it satisfies
\begin{equation}\label{tangentZ}
T \subset N_{R^{1/2+\delta'}}Z \cap B_R, \quad
\Angle \big( D(T),T_zZ(P)\big)
\le R^{-1/2+\delta'}
\end{equation}
for all nonsingular points $z\in N_{2R^{1/2+\delta'}}(T)
\cap 2B_R \cap Z$. The collection of tubes that are $R^{-1/2+\delta'}$-tangent to $Z$ 
is denoted by $T_Z(R^{-1/2+\delta'})$. We say that $f$ is concentrated on wave packets from $T_Z(R^{-1/2+\delta'})$ if 
\[
\sum_{T \notin T_Z(R^{-1/2+\delta'})} \| f_T \|_2 =O( R^{-N}) \|f\|_2
\]
holds for sufficiently large $N>0$.

\noindent{\it Wall part.} 
Now we consider the case $\mathcal I \neq \widetilde {\mathcal I}$; hence, we can choose $i_\circ \in \mathcal I \setminus \widetilde{\mathcal I}$. From \eqref{d3}, 
\begin{align*}
\| U_\alpha f \|_{L_x^pL_t^q(B_R)}^p
	&\le C_1D^3
\| \raisebox{.5ex}{$\chi_{\widetilde O_{i_\circ}}$} U_\alpha f \|_{L_x^pL_t^q(B_R)}^p +
C_1D^3\| \raisebox{.3ex}{$\chi_{O_{i_\circ} \cap W}$}  U_\alpha f  \|_{L_x^pL_t^q(B_R)}^p.
\end{align*}
Since $i_\circ \in \mathcal I \setminus \widetilde{\mathcal I}$, we have
$C_1D^3\| \raisebox{.5ex}{$\chi_{\widetilde O_{i_\circ}}$} U_\alpha f \|_{L_x^pL_t^q(B_R)}^p\le 2^{-1}\| U_\alpha f \|_{L_x^pL_t^q(B_R)}^p$. Therefore,
\begin{align*}
\| U_\alpha f \|_{L_x^pL_t^q(B_R)}^p \le 2C_1D^3  \| \raisebox{.3ex}{$\chi_{W}$}  U_\alpha f \|_{L_x^pL_t^q(B_R)}^p. 
\end{align*}
Thus, it suffices to consider the wave packets concentrated on the wall.

Recalling that $ \supp \widehat f\subset B^2(\xi_0,r)$, for $1 \ll K \ll R^{\epsilon}$, we cover $B^2(\xi_0,r)$ by boundedly overlapping collection of balls $\omega$ of radius $K^{-1}r$ and let $f = \sum_{\omega} f_{\omega}$, where $\widehat{f_\omega}$ is supported on $\omega$. For each fixed $B_j$, we set 
\[
f_{\omega,j,\mathrm{tang}}=(f_\omega)_{j,\mathrm{tang}}, \qquad f_{\omega,j,\mathrm{trans}}=(f_\omega)_{j,\mathrm{trans}}.
\]
We define a bilinear tangential operator by
\[
Bil(U_\alpha f_{j,\mathrm{tang}})(x,t) = \sum_{\dist(\omega_1,\omega_2) \ge K^{-1}r} 
| U_\alpha f_{\omega_1,j,\mathrm{tang}}(x,t)|^{1/2} |U_\alpha f_{\omega_2,j,\mathrm{tang}}(x,t)|^{1/2}.  
\]

We set
\begin{align}\label{Ball}
B= \{ (x,t) \in B_R : K^{\epsilon^3} \max_\omega | U_\alpha f_\omega(x,t)|  \le |U_\alpha f(x,t)| \,\}
\end{align}
and for $(x,t)\in W\cap B$,
\[
\Omega= \{ \omega : | U_\alpha f_{\omega,j,\mathrm{tang}}(x,t)| \le K^{-4} | U_\alpha f(x,t) | \,\}.
\]
By fixing $B_j$ and $(x,t) \in B_j \cap W \cap B$, we first consider the case in which all balls $\omega$ in $\Omega^\mathsf{c}$ are adjacent. 
Hence, $\#\Omega^\mathsf{c} \lesssim 1$,
and it follows that $ \sum_{\omega \in \Omega^\mathsf{c}} | U_\alpha f_\omega(x,t)| \le \frac 12 | U_\alpha f (x,t)|$
by \eqref{Ball}.
Thus, we have $\frac 12 | U_\alpha f(x,t)|  \le |\sum_{\omega \in \Omega} U_\alpha f_\omega(x,t)| $. Then,
\begin{align*}
\frac 12 |U_\alpha f(x,t)|
	&\le \big|\sum_{\omega\in \Omega} U_\alpha f_{\omega,j,\mathrm{trans}}(x,t) \big| 
	+ \big|\sum_{\omega \in \Omega} U_\alpha f_{\omega,j,\mathrm{tang}}(x,t) \big| +O(R^{-N})\|f\|_2.
\end{align*}
Since the total number of $\omega$ is $\le 10K^2$, we get
\begin{align*}
\frac 12 |U_\alpha f(x,t)| 
	&\lesssim \big|\sum_{\omega\in \Omega} U_\alpha f_{\omega,j,\mathrm{trans}}(x,t) \big| + K^{-2} |U_\alpha f(x,t)|
+O(R^{-N})\|f\|_2.
\end{align*} 
Otherwise, for $(x,t)\in B_j \cap W \cap B$, there are $\omega_1,\omega_2 \in \Omega^\mathsf{c}$ such that dist$(\omega_1,\omega_2) \gtrsim K^{-1}r$. Then, by the definition of $\Omega$,
$| U_\alpha f(x,t)| \le K^{4} Bil( U_\alpha f_{j,\mathrm{tang}})(x,t)$. Therefore, we have the following.
\begin{lem}\label{decomposition}
For each point $(x,t) \in W \cap B_R$, 
there exists a collection $\Omega$ of balls $\omega$ of radius $K^{-1}r$ such that
\begin{align*}
| \raisebox{.3ex}{$\chi_W$} U_\alpha f(x,t)|^p 
	&\lesssim 
| \raisebox{.3ex}{$\chi_{W\cap B^\mathsf{c}} $}U_\alpha f(x,t) |^p +\sum_j \big|\sum_{\omega \in \Omega} 
\raisebox{.3ex}{$ \chi_{W \cap B_j}$} U_\alpha f_{\omega,j,\mathrm{trans}}(x,t) \big|^p\\
	&+ \sum_j K^{4p} \big| \raisebox{.3ex}{$\chi_{W\cap B_j}$} Bil( U_\alpha f _{j,\mathrm{tang}})(x,t)\big|^p + O(R^{-N})\|f\|_2^p.\nonumber
\end{align*}
\end{lem}
By Lemma \ref{decomposition}, we have
\begin{equation}
\begin{aligned}\label{wall123}
\| U_\alpha f \|_{L_x^pL_t^q(W\cap B_R)} ^p
	&\lesssim  
\big\| U_\alpha f \big\|_{L_x^pL_t^q(W\cap B^\mathsf{c})}^p
+\sum_j \big\|\sum_{\omega \in \Omega} U_\alpha f_{\omega,j,\mathrm{trans}} \big\|_{L_x^pL_t^q(W \cap B_j)}^p \\
	&+ 
\sum_j \big\| K^{4}Bil( U_\alpha f_{j,\mathrm{tang}}) \big\|_{L_x^pL_t^q(W\cap B_j)}^p
+ O(R^{-N})\|f\|_2^p.
\end{aligned}
\end{equation}
From \eqref{Ball}, the first term on the right-hand side of \eqref{wall123} is bounded by
\begin{equation*}
\|  U_\alpha f \|_{L_x^pL_t^q(W\cap B^\mathsf c)}^p 
	\le K^{\epsilon^3p} \sum_{\omega} \| U_\alpha f_\omega \|_{L_x^pL_t^q(B_R)}^p.
\end{equation*}
By applying the induction hypothesis \eqref{main ineq} to the right-hand side,  we have
\[
\|  U_\alpha f \|_{L_x^pL_t^q(W\cap B^\mathsf c)}^p 
\le
K^{\epsilon^3p}
\sum_{\omega} \big[ C_\epsilon (K^{-1}r)^{\epsilon^2}R^\epsilon\| f_\omega \|_2  \big]^p \le 10K^{(\epsilon^3-\epsilon^2)p} \big[ C_\epsilon r^{\epsilon^2} R^\epsilon \| f \|_2  \big]^p.
\]
Since $ K \ll R^{\epsilon}$ and $R \gg 1$, we have $10K^{\epsilon^3-\epsilon^2} \le 1/6$. This completes the induction step for the first term on the right-hand side of \eqref{wall123}.
In the remainder of this section, we treat the second and third terms on the right-hand side of \eqref{wall123}.

\noindent{\it Transversal case.} 
Now, we deal with the transversal term in \eqref{wall123}. In \cite{G}, it was shown that
for each tube $T \in \mathcal T$,
\begin{align}\label{algebraic}
\#\{j :T \in \mathcal T_{j,\mathrm{trans}}\} \le R^{O(\epsilon^4)}.
\end{align}	
 Since $\# \Omega \le 2^{10K^2}$, the second term on the right-hand side of \eqref{wall123} is controlled by
\begin{align}\label{maxJ}
\sum_j
\big\| \max_{\Omega} \big|\sum_{\omega \in \Omega} U_\alpha f_{\omega,j,\mathrm{trans}}\big| \big\|_{L_x^pL_t^q(W\cap B_j)}^p 
	\le \sum_j 2^{10K^2} \big\| \sum_{\omega \in \Omega} U_\alpha f_{\omega,j,\mathrm{trans}} \big\|_{L_x^pL_t^q(B_j)}^p.
\end{align}
By applying the induction hypothesis to a ball $B_j$ of radius $R^{1-\delta}$, the right-hand side of \eqref{maxJ} is bounded by $\sum_j 2^{10K^2} \big[ C_\epsilon r^{\epsilon^2}R^{(1-\delta)\epsilon} \| f_{j,trans}\|_2\big]^p$.
Therefore, by applying \eqref{algebraic}, we get
\begin{align*}
\sum_j \big\| \max_{\Omega} \big|\sum_{\omega \in \Omega} U_\alpha f_{\omega,j,\mathrm{trans}}\big| \big\|_{L_x^pL_t^q(W\cap B_j)}^p 
\le
R^{O(\epsilon^4)} 2^{10K^2} R^{-\delta\epsilon p} \big[ C_\epsilon r^{\epsilon^2}R^{\epsilon} \| f \|_2\big]^p.
\end{align*}
Since $K \ll R^\epsilon$, we take $\delta=\epsilon^2$ and obtain $2^{10K^2} R^{O(\epsilon^4)-\epsilon^3 p}\le 1/6$ for sufficiently large $R>0$. Therefore, the induction closes for the transversal term.

\noindent{\it Bilinear tangential case.}
To estimate the third term on the right-hand side of \eqref{wall123}, it remains to prove the following bilinear maximal estimates. 
\begin{thm}\label{BT}
For any $\epsilon>0$ and $p>3$, there exists $C_\epsilon>0$ such that	
\[ \Big( \int_{B_R} \sup_{t : (x,t) \in W\cap B_j} | Bil( U_\alpha f_{j,\mathrm{tang}})(x,t) |^p \,dx \Big)^{1/p}
	\le C_\epsilon R^{\epsilon/2} \|f\|_2.
\]
\end{thm}

Indeed, assuming Theorem \ref{BT}, by H\"older's inequality, it follows that for $q >\epsilon^{-4}$,
\begin{align*}
\sum_jK^{4p}\|  Bil( U_\alpha f_{j,\mathrm{tang}}) \|_{L_x^pL_t^q(W \cap B_j )}^p
	&\le \sum_jK^{4p}R^{\epsilon^4p} \|  Bil( U_\alpha f_{j,\mathrm{tang}}) \|_{L_x^pL_t^\infty(W \cap B_j)}^p\\
	&\le R^{3\delta} K^{4p} R^{\epsilon^4p} [C_\epsilon R^{\epsilon/2} \|f\|_2]^p
\end{align*}
since the number of $j$ is $\lesssim R^{3\delta}$. Because $\delta=\epsilon^2$, $ K \ll R^\epsilon$ and $r \ge R^{-1/2}$, we obtain $ K^4R^{\epsilon^4+\epsilon/2+3\delta/p}\le 1/6 R^\epsilon r^{\epsilon^2}$. This completes the proof of Theorem \ref{main result}.

To prove Theorem \ref{BT}, we show the following maximal estimate.

\begin{thm}\label{8.1}

Let $0< r \le 1$, $\xi_0  \in \mathbb A_1$, and $K=K(\epsilon)$ be a sufficiently large constant. Suppose that 
the supports of $ \widehat f $ and  $\widehat g $ are contained in $B^2(\xi_0, r)$ and separated by $K^{-1}r$. If $f,g$ are concentrated on the wave packets from $T_{ Z}(R^{-\frac12+\delta'})$ with $\delta' \le 100\delta$,
then
there exist constants $c$ and $C$ such that
\begin{equation}\label{bimax}
\big\| |U_\alpha f|^{\frac12}| U_\alpha g|^{\frac12} \big\|_{L_x^3L_t^\infty(B_R)}
	\le CR^{c\delta'}\|f\|_2^{\frac12} \|g\|_2^{\frac12}.
\end{equation}
\end{thm}

In Theorem \ref{BT}, we are concerned with the function $f_{j,\mathrm{tang}}$ defined on a smaller ball $B_j$ of radius $R_1=R^{1-\delta}$. We can easily see that the wave packets of $f_{j,\mathrm{tang}}$ on the ball $B_j$ are
concentrated in $T_Z(R_1^{-\frac 12+\delta'})$ for some $\delta'\le 100\delta$ (see \cite{G2}), and we omit the details. Since $R_1^{\delta'} \le R^{c\delta}$, we obtain the desired bound in Theorem \ref{BT} from the estimate \eqref{bimax}.

\section{Proof of Theorem \ref{8.1}}\label{sec3}
In this section, we prove Theorem \ref{8.1} by considering linear and bilinear refined Strichartz estimates (Propositions \ref{linear} and \ref{bilinear}, respectively), which are variants of the estimates presented in \cite{DGL} for a class of elliptic functions.
We begin by defining a class of elliptic phase functions.

\subsection{Class of elliptic functions}
We consider the class of phase functions that are small perturbations of $\phi_0(\xi)=|\xi|^2/2$.

\begin{defn}
Let $0<\epsilon_0 \ll 1$, $\rho>0$, and $n \ge 10^3$ be a positive integer. We define a class of normalized phase functions by
\begin{align*}
\mathcal P(\epsilon_0,n) = \{ \phi \in \mathcal C^n (B^2(0,2)) : \| \phi-\phi_0 \|_{\mathcal C^n(B^2(0,2))} \le \epsilon_0 \}.
\end{align*}
Let $\phi \in \mathcal P(\epsilon_0,n)$,  $\xi_0 \in  \mathbb A_1$, and $\mathrm H\phi(\xi_0)$ be the Hessian matrix of $\phi$ at $\xi=\xi_0$. Then,  $\mathrm H\phi$ is positive definite on $B^2(0,2)$ and $\mathrm H\phi(\xi_0) = \mathrm T^{-1} \mathrm D \mathrm T$, where $\mathrm D$ is a diagonal matrix with eigenvalues $\lambda_1>0$ and $\lambda_2>0$, and $\mathrm T$ is a symmetric matrix. If we set $\mathrm H_{\xi_0} :=\sqrt{\mathrm H\phi(\xi_0)}$, then $\mathrm H_{\xi_0}=\mathrm T^{-1} \mathrm D^{1/2} \mathrm T$ with $\mathrm D^{1/2} = (\sqrt{\lambda_1} e_1, \sqrt{\lambda_2 }e_2)$. We denote the normalization of $\phi$ by
\begin{equation}\label{phi}
\phi_{\xi_0}^\rho(\xi)
	=\rho^{-2} \big(\phi(\rho \mathrm H_{\xi_0}^{-1}\xi+\xi_0)
-\phi(\xi_0) - \rho \nabla\phi(\xi_0)\cdot \mathrm H_{\xi_0}^{-1} \xi \big).
\end{equation}
\end{defn}
Then, we observe the following, which plays an important role in the induction argument.
We denote by $\interior \mathbb A_r$ the interior of $\mathbb A_r$.
\begin{lem}\label{reducephi}
Let $\epsilon_0>0$ and $\xi_0 \in  \interior\mathbb A_1$. Suppose that $\phi \in \mathcal C^n(\mathbb A_1)$ and the Hessian matrix of $\phi$ is positive definite.  Then, there exists a constant $\rho_0>0$ such that $\phi_{\xi_0}^\rho \in \mathcal P(\epsilon_0,n)$ whenever $\rho \le \rho_0$. Moreover, if $\phi \in \mathcal P(\epsilon_0,n)$, then for sufficiently small $\epsilon_0>0$, there exists a constant $\rho_1$ such that $\phi_{\xi_0}^\rho \in \mathcal P(\epsilon_0,n)$ whenever $\rho\le \rho_1$.
\end{lem}

\begin{proof}
By \eqref{phi} and Taylor's expansion, we may write 
\[
\phi_{\xi_0}^\rho(\xi) =|\xi|^2/2+ \mathcal E(\xi,\xi_0,\rho)
\]
where
$ \| \mathcal E(\cdot, \xi_0,\rho)\|_{\mathcal C^n(\mathbb A_1)} =O(\rho |\mathrm H_{\xi_0}^{-1}\xi|^3)$.
Thus, we can take $\rho_0$ such that
$\| \phi_{\xi_0}^\rho - \phi_0 \|_{\mathcal C^n(\mathbb A_1)} \le C\rho \le \epsilon_0$ holds for any $\rho \le \rho_0$.
Similarly, if $\phi \in \mathcal P(\epsilon_0,n)$ and $\xi_0 \in B^2(0,1)$, then we can take $\rho_1>0$ such that $\phi_{\xi_0}^\rho \in \mathcal P(\epsilon_0,n)$ whenever $\rho \le \rho_1$.
\end{proof}

Suppose that $\phi \in \mathcal C^n(\mathbb A_1)$ and the Hessian matrix of $\phi$ is positive definite. For a given small $\epsilon_0>0$, by partitioning $B^2(0,1)$ into smaller balls of radius $\rho_0$
and applying Lemma \ref{reducephi}, $\phi_{\xi_0}^\rho\in
\mathcal P(\epsilon_0,n)$ for any $\rho \le \rho_0$.  Therefore, hereafter, we may fix $n \ge 10^3$ and simply denote $\mathcal P(\epsilon_0,n)$ by $\mathcal P(\epsilon_0)$. To prove Theorem \ref{8.1}, it suffices to consider $\phi\in \mathcal P(\epsilon_0)$.

\subsection{Parabolic rescaling}
We define a linear map $\mathcal A_{\xi_0}^\rho:\mathbb R^3 \rightarrow \mathbb R^3$ by
\begin{align}\label{map}
(\mathcal A_{\xi_0}^\rho)^{-1}(x,t)=
(\rho^{-1} \mathrm H_{\xi_0}^{t} x- \rho^{-2}t \nabla \phi(\xi_0), \rho^{-2}t).
\end{align}

\begin{lem}\label{lem:scale}
Let $\epsilon_0>0$ be sufficiently small, $\xi_0 \in \interior \mathbb A_1$, and $\phi \in \mathcal P(\epsilon_0)$.
Suppose that $\widehat f$ is supported on a ball $B^2(\xi_0, \rho)$ 
for $\rho \le \rho_1$, where $\rho_1$ is given in Lemma \ref{reducephi}. Then, there exist $\widetilde{f}$, $\widetilde T$, and a constant $C=C(\xi_0,\rho)$ such that
\begin{equation}\label{frho}
\| e^{it\phi} f \|_{L^q(T)} 
	=C\rho^{1-\frac{4}q} \| e^{it \phi_{\xi_0}^\rho} \widetilde f \|_{L^q(\widetilde T)},
\end{equation}
where $\mathcal F(\widetilde f)$ is supported on $B^2(0,1)$ such that  $\| \widetilde f\|_2 =\|f\|_2$, and
\[
\widetilde T=\{(x,t): (\mathcal A_{\xi_0}^\rho)^{-1} (x,t)\in T \}.
\]
\end{lem}

\begin{rmk}
Let $\phi \in \mathcal P(\epsilon_0)$. 
Suppose that $T$ is a tube of dimensions $\rho^{-1}M \times \rho^{-1}M \times \rho^{-2}M$ centered at the origin with its long axis parallel to $(-\nabla\phi(\xi_0),1)$, i.e.,
\[
T=\{ (x,t) : |x+t\nabla \phi(\xi_0)| \le \rho^{-1}M, |t| \le \rho^{-2}M\}.
\]
Then, $\widetilde T
=\{(x,t) : | \rho^{-1} \mathrm H_{\xi_0}x| \le \rho^{-1}M, |\rho^{-2}t | \le \rho^{-2}M \}$
and $\widetilde T$ is contained a cube of side length $CM$
for some $C=C(\epsilon_0)$.
\end{rmk}

\begin{proof}[Proof of Lemma \ref{lem:scale}]
Let $| \mathrm H|$ denote the determinant of a matrix $\mathrm H$.
By change of variables $\xi \rightarrow \rho \mathrm H_{\xi_0}^{-1}\xi+\xi_0$, we have
\[
|e^{it\phi}f(x)|
	=\rho | \mathrm H_{\xi_0}^{-1}|^{\frac 12} \Big|\int e^{i (x,t)\cdot 
	(\rho \mathrm H_{\xi_0}^{-1}\xi, \phi(\rho \mathrm H_{\xi_0}^{-1}\xi+\xi_0))} \rho |\mathrm H_{\xi_0}^{-1}|^{\frac 12}\widehat f(\rho \mathrm H_{\xi_0} ^{-1}\xi+\xi_0)\,d\xi \Big|.
\]
We define $\widetilde f$ by 
\begin{align}\label{tildef}
\mathcal F \widetilde f(\xi) =\rho |\mathrm H_{\xi_0}^{-1}|^{\frac 12} \widehat {f}(\rho \mathrm H_{\xi_0}^{-1}\xi+\xi_0).
\end{align}
Then, $\widetilde f $ has a Fourier support on $B^2(0,1)$ and
 $\| \widetilde f\|_2 = \| f \|_2$. By the definition of $\phi_{\xi_0}^\rho$ (\eqref{phi}), we note that
\[
\big(\rho^{-1}\mathrm H_{\xi_0}^tx-\rho^{-2} t \nabla \phi(\xi_0) \big)
\cdot \rho\mathrm H_{\xi_0}^{-1}\xi
+\rho^{-2}t\big( \phi(\rho \mathrm H_{\xi_0}^{-1}\xi+\xi_0) - \phi(\xi_0) \big)
=x \cdot\xi +  t \phi_{\xi_0}^\rho(\xi).
\]
Thus, by change of variables $x\rightarrow \rho^{-1} \mathrm H_{\xi_0}^{t}x-t \nabla \phi(\xi_0)$ and $t \rightarrow \rho^{-2}t$ (i.e., $(x,t) \rightarrow (\mathcal A_{\xi_0}^\rho)^{-1}(x,t)$), the desired bound \eqref{frho} follows by taking $C=|\mathrm H_{\xi_0}|^{1/q-1/2}$.
\end{proof}

Now, we observe that the condition \eqref{tangentZ} is preserved after parabolic rescaling.

\begin{lem}\label{tan_phirho}
Let $0<\rho \le \rho_0$ and $\xi_0 \in  \interior \mathbb A_1$. Suppose that $f$ is concentrated on the wave packets
from $T_{Z}(R^{-1/2+\delta'})$ associated with $\phi$ for some $Z=Z(P)$. 
If $\widetilde f$ is given by \eqref{tildef}, then $\widetilde f$ is concentrated on the wave packets from $T_{Z'}(\rho^{-1}R^{-1/2+\delta'})$ associated with $\phi_{\xi_0}^\rho$ and $Z'=Z'(\widetilde P)$, where $\widetilde P=P \circ (\mathcal A_{\xi_0}^\rho)^{-1}$.
\end{lem}

\begin{proof}
It suffices to show that
\begin{equation}\label{angle12}
|(-\nabla_{\xi'} \phi_{\xi_0}^\rho(\xi') ,1) 
\cdot \nabla_{x',t'} \widetilde P| /| \nabla_{x',t'}\widetilde P | \lesssim \rho^{-1}R^{-1/2+\delta'},
\end{equation}
where $\xi'= \rho^{-1} \mathrm H_{\xi_0}(\xi-\xi_0)$ and  $(x',t') = \mathcal A_{\xi_0}^\rho(x,t)$. Since $\widetilde P(x',t')=P(\rho^{-1} \mathrm H_{\xi_0}^{t} x'- \rho^{-2}t' \nabla \phi(\xi_0), \rho^{-2}t')$,
we have
\begin{align}\label{QQQ}
\begin{cases}
\nabla_{x'}\widetilde P= \rho^{-1} \mathrm H_{\xi_0} \nabla_x P, \\
\partial_{t'}\widetilde P=\rho^{-2} \big(\partial_t P - \nabla_xP \cdot \nabla \phi(\xi_0)).
\end{cases}
\end{align}
Combining this with $\nabla_{\xi'} \phi_{\xi_0}^\rho(\xi') =\rho^{-1} \mathrm H_{\xi_0}^{-t} ( \nabla \phi(\xi) - \nabla \phi(\xi_0))$, we get
\[
(-\nabla_{\xi'} \phi_{\xi_0}^\rho(\xi') ,1)  \cdot (\nabla_{x'}\widetilde P, \partial_{t'} \widetilde P)=\rho^{-2} (-\nabla \phi(\xi), 1) \cdot (\nabla_{x}P, \partial_t P).
\]
On the other hand, from \eqref{QQQ}, we can easily deduce that $|\nabla_{x',t'}\widetilde P|\gtrsim \rho^{-1}|\nabla_{x,t}P| $
by considering the cases $ |\partial_t P| \ge 2|\nabla_xP\cdot \nabla \phi(\xi_0)|$ and $ |\partial_t P| \le 2|\nabla_xP\cdot \nabla \phi(\xi_0)|$ separately.
By the assumption
$|(-\nabla\phi,1)\cdot \nabla_{x,t}P | /|\nabla_{x,t}P| \lesssim R^{-1/2+\delta'}$,
the desired estimate \eqref{angle12} follows.
\end{proof}

\subsection{Linear refined Strichartz estimates}
Before proving Theorem \ref{8.1}, we consider
the linear and bilinear refined Strichartz estimates.
We first prove the linear refined Strichartz estimates (Proposition \ref{linear})
and then prove the bilinear estimates (Proposition \ref{bilinear}) by using Proposition \ref{linear}.

\begin{prop}\label{linear}
Let $\phi \in \mathcal P(\epsilon_0)$. Suppose that $f$ is concentrated on the wave packets from $T_Z(R^{-1/2+\delta'})$ and $\widehat f$ is supported on $B^2(0,1)$. Let $Q_1,Q_2,\dots$ be lattice cubes of side length $R^{1/2}$ in $B^3(0,R)$. Suppose that $M$ cubes $Q_j$ are contained in $B^2(0,R) \times [t_0 , t_0+R^{1/2}]$
for each $t_0 \in R^{1/2}\mathbb Z \cap [0,R]$, and for each $Q_j$, 
\begin{equation}\label{qjconst}
\| e^{it \phi} f \|_{L^6(Q_j)} \quad~ \text{is essentially constant}.
\end{equation}
Then, for any $\epsilon>0$, there exist constants $C_\epsilon,C \ge 1$ such that
\begin{equation}\label{unionQ}
\| e^{it \phi} f \|_{L^6(\cup_j Q_j)}
\le C_\epsilon R^{-1/6+\epsilon+C\delta'} M^{-1/3} \|f\|_2.
\end{equation}
\end{prop}

We start by recalling the $l^2$-decoupling inequality for an elliptic paraboloid, which was obtained by Bourgain and Demeter \cite{BD}.

\begin{thm}\label{BDlower}
Suppose that $\widehat g$ is supported in a $\sigma$-neighborhood of an elliptic paraboloid $S$ in $\mathbb R^2$. Let $\tau$ be rectangles of dimensions $\sigma^{1/2}\times \sigma$, which cover a $\sigma$-neighborhood of $S$. If $\widehat{g}_\tau=\widehat g \chi_\tau$, then for $\epsilon>0$ and $2 \le p \le 6$, we have
\[
\| g \|_{L^p(\mathbb R^2)}
\le C_\epsilon \sigma^{-\epsilon} \big( \sum_\tau \| g_\tau \|_{L^p(\mathbb R^2)}^2 \big)^{1/2}.
\]
\end{thm}

Let us consider the wave packet decomposition
\begin{align}\label{FT}
f=\sum_{T} f_{T}
\end{align}
such that the Fourier support of $ f_{T}$ is contained in a ball of radius $R^{-1/4}$
and $f_{T}$ is essentially supported on a ball of radius $R^{3/4}$.
Then, $e^{it\phi} f_{T}$ restricted to the ball $B^3(0,R)$ is essentially supported on a tube $T$ of dimensions $R^{3/4} \times R^{3/4} \times R$.
Since $f$ is concentrated on the wave packets from $T_Z(R^{-1/2+\delta'})$
for some $Z$, we can apply Theorem \ref{BDlower} and obtain the following.

\begin{prop}\label{decouple}
Let $\phi \in \mathcal P(\epsilon_0)$ for sufficiently small $\epsilon_0>0$.
Let $f=\sum_T f_T$ be as stated above. Suppose that $f$ is concentrated on the wave packets from $T_Z(R^{-1/2+\delta'})$ for some $Z=Z(P)$ and suppose that $Q$ is a cube of side length $R^{1/2}$ contained in the $2R^{1/2+\delta'}$-neighborhood of $Z$. Then, 
\begin{align}\label{conf-dcp}
\| e^{it\phi}f \|_{L^6(Q)} \le C_\epsilon R^{\epsilon}
\big( \sum_{T} \| e^{it\phi} f_{T} \|_{L^6(\omega_{Q})}^2 \big)^{1/2}
+O(R^{-N})\|f\|_2.
\end{align}
Here, $\omega_Q(z)=(1+R^{-\frac 12}|z-c_Q|)^{-100}$, where $c_Q$ is the center of the cube $Q$.
\end{prop}

\begin{proof}
Let $\psi\in \mathcal S(\mathbb R^3)$ such that $\psi=1$ on $B^3(0,1)$ and $\widehat \psi$ is supported on $B^3(0,1)$.  By letting $\psi_Q(z)=\psi(CR^{-\frac 12}(z-c_Q))$ for some constant $C>0$ such that $\psi=1$ on $Q$, we get
\[
\| e^{it\phi}f \|_{L^6(Q)}
\le \|  \psi_{Q} \,e^{it\phi} f \|_{L^6(\mathbb R^3)}.
\]
Since $f$ is concentrated on the wave packets from $T_Z(R^{-1/2+\delta'})$, it suffices to consider
the wave packets that are contained in the $R^{\frac 12+\delta'}$-neighborhood of a plane $W=T_zZ(P)$.
We claim that
\begin{align}\label{dcp-plane}
\|\psi_{Q}  e^{it\phi} f \|_{L^6(W)}
\le C_{\epsilon} R^{\frac \epsilon2}
\big(\sum_{T} \|   \psi_{Q} e^{it\phi} f_{T} \|_{L^6(W)}^2 \big)^{1/2}
+O(R^{-N})\|f\|_2.
\end{align}
By assuming \eqref{dcp-plane}, we prove \eqref{conf-dcp}. By integrating along the $W^\perp$ axis and using Minkowski's inequality and Fubini's theorem, we obtain
\begin{align*}
\| e^{it\phi} f \|_{L^6(Q)}
&\le C_{\epsilon} R^{\frac \epsilon 2} \big(\sum_{T} \|  \psi_{Q}
e^{it\phi} f_{T} \|_{L^6(\mathbb R^{3})}^2 \big)^{1/2}+O(R^{-N})\|f\|_2.
\end{align*}
Here, we use the fact that the number of tubes $T$ intersecting $Q$ is $\lesssim R^{\epsilon/100}$. Since $\psi_Q$ decays rapidly outside $Q$, we get the desired result \eqref{conf-dcp}.
	
Now, we prove \eqref{dcp-plane}. It suffices to show that the restriction of $\mathcal F(\psi_{Q} e^{it\phi} f)$  to $W$ is contained in an $R^{-1/4}$-neighborhood of an elliptic paraboloid. 
Since $\mathcal F(\psi_Q)$ is supported on $B^3(0,R^{-1/2})$, it suffices to consider the restriction of $\mathcal F(e^{it\phi}f)$ to $W$. Let $\mathbf n$ be the unit normal vector of $W$.
Since $\widehat f$ is supported on $B^2(0,1)$, we have $|\mathbf n \cdot e_3| < 1/2$.  By rotation and dilation, we may assume that $\mathbf n=\mathbf n'/|\mathbf n'|$, where $\mathbf n'=(0,1,n_3)$ for some $|n_{3}| \lesssim 1$. 
Since $\partial_{\xi_2}^2\phi \neq 0$ on $B^2(0,1)$, by the implicit function theorem, there is a function $g \in \mathcal C_0^{1}((-1,1))$ such that 
\begin{align}\label{eqeq1}
(-\nabla \phi(\xi_1, g(\xi_1)),1) \cdot \mathbf n=0,
\end{align}
and equivalently,
$n_3=\partial_{\xi_2} \phi (\xi_1,g(\xi_1))$.

Note that $|(-\nabla \phi(\xi),1) \cdot \mathbf n|=|-\partial_{\xi_2}\phi+n_3| \lesssim R^{-1/4}$
on the support of $\supp \widehat{f}$. Since $\partial_{\xi_2}^2 \phi \neq0$, by the mean value theorem, we have
\begin{align}\label{n3}
|\xi_{2}-g(\xi_1)| \lesssim R^{-1/4}.
\end{align}
Therefore, we may write
\begin{align*}
(\xi,\phi(\xi))= (\xi_1,0,\widetilde \phi(\xi_1))+\xi_2 \mathbf n' +\mathcal E(\xi)e_3,
\end{align*}
where $\widetilde \phi(\xi_1)=\phi(\xi_1, g(\xi_1))-n_3g(\xi_1)$ and
$\mathcal E(\xi)= \phi(\xi)-\phi(\xi_1,g(\xi_1))-n_3(\xi_{2}-g(\xi_1))$.
By \eqref{n3} and the mean value theorem, it is easy to see that $\mathcal E(\xi)=O(R^{-1/2})$. Moreover, 
by \eqref{eqeq1}, $\partial_{\xi_1}\widetilde \phi(\xi_1)
=(\partial_{\xi_1} \phi) (\xi_1, g(\xi_1))$; hence,
$\partial_{\xi_1}^2\widetilde \phi(\xi_1) =(\partial_{\xi_1}^2\phi) (\xi_1,g(\xi_1)) +O(\epsilon_0)$	
whenever $\phi \in \mathcal P(\epsilon_0)$ for sufficiently small $\epsilon_0>0$.
Thus, the curve parameterized by $\xi_1 \rightarrow \widetilde \phi(\xi_1)$  is contained in an $R^{-1/2}$-neighborhood of the elliptic paraboloid. By Theorem \ref{BDlower}, we obtain the desired bound \eqref{dcp-plane}.
\end{proof}

Using Proposition \ref{decouple}, we can prove Proposition \ref{linear}
by induction.

\begin{proof} [Proof of Proposition \ref{linear}]
Let us set $\mathcal A_0:=\mathcal A_{\xi_0}^{R^{-1/4}}$ for simplicity. Recalling \eqref{FT},
for each tube $T$, we choose a cube $Q_T$ of side length $R^{1/2}$ that is contained in $\mathcal A_0( T)$.
We decompose $Q_T$ into horizontal strips $S'$ of height $R^{1/4}$ in $Q_T$ such that $Q_T=\cup S'$.
Furthermore, each strip $S'$
is decomposed into cubes $Q'$ of side length $R^{1/4}$.
We set $T' := \mathcal A_0^{-1}(Q')$, which is a tube of size $\sim R^{1/2}\times R^{1/2}\times R^{3/4}$
such that the union of all $T'$ covers $T$ (see Figure \ref{figure1}).
	
\begin{figure}
\centering	\includegraphics[scale=0.45]{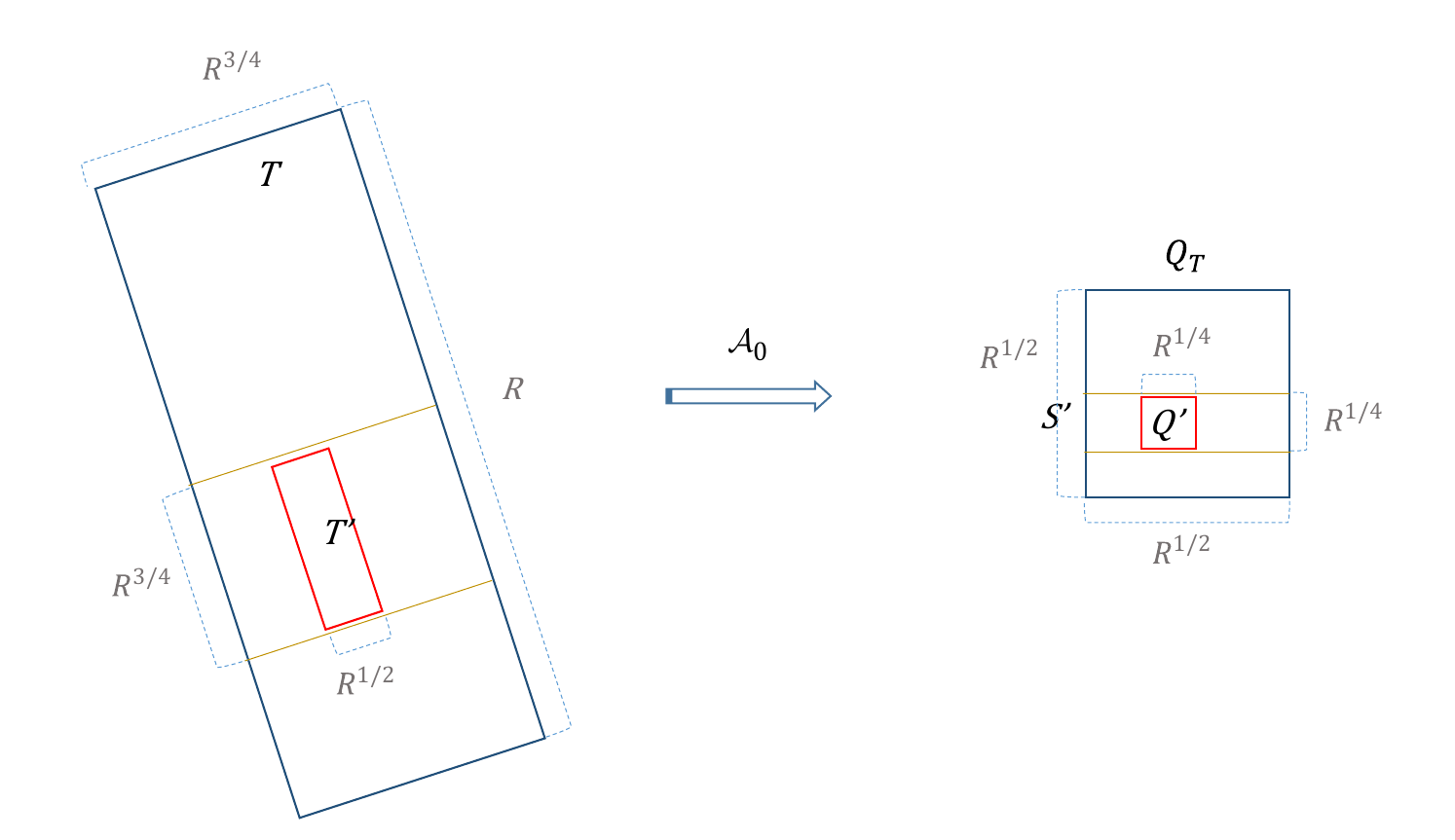}
\caption{}
\label{figure1}
\end{figure}

By dyadic pigeonholing on the size of $\|f_T\|_2$, we may assume that $\| f_T\|_2$ is essentially constant for each tube $T$. In fact, the number of $h$ is only $C\log R$ such that $\|f_T\|_2 \sim h$
since it is negligible if $\| f_T\|_2 \le R^{-N}\|f\|_2$.
Therefore, we choose one of $h$ such that $\|f_T\|_2$ is essentially constant for a fraction $\sim 1/(\log R)$
of tubes $T$. For  a fixed such tube $T $, we again perform dyadic pigeonholing on the size of $\|e^{it\phi} f_T \|_{L^6(T')}$ such that 
\begin{align}\label{t'} 
\|e^{it\phi} f_T \|_{L^6(T')} \quad~ \text{is essentially constant,}	
\end{align}
which is greater than $R^{-N}\|f\|_2$ since this part can be absorbed in the error term in \eqref{conf-dcp} if the constant is less than  $R^{-N}\|f\|_2$. We sort $T'$ further according to the number of $T'$ arranged along the short axis of $T$. Precisely, we may assume that
\[
\#\{ T' :T' \subset \mathcal A_0^{-1}(S') \} \sim M'
\]
for a dyadic number $M'$. For simplicity, by abuse of notation, we denote such tubes by $T'$. Then, it follows that
\begin{equation}\label{Bbox}
\| e^{it\phi} f \|_{L^6(Q_j)}
\lesssim (\log R)^3 \big\| \sum_{T} \sum_{T'} \chi_{\cup T'}  e^{it\phi} f_{T} \big\|_{L^6(Q_j)}
\end{equation}
for a fraction $\sim 1/(\log R)^3$ of all cubes $Q_j$ in $\cup_j Q_j$.
Finally, we sort the cubes $Q_j$ further such that $Q_j$ satisfies \eqref{Bbox}  and
each $Q_j$ is contained in $\sim \mu$ tubes $T$ such that
$Q_j \subset T' \subset T$.
Let $\mathcal Q$ denote the set of such cubes $Q_j$; then, we see that $\# \mathcal Q  \gtrsim (\log R)^{-4}M$ by dyadic pigeonholing. By \eqref{qjconst}, we have
$\| e^{it\phi} f\|_{L^6(\cup_j Q_j) }^6
\lesssim (\log R)^4 \| e^{it \phi} f\|_{L^6(\cup_{Q_j \in \mathcal Q}Q_j)}^6.
$
Hence, it suffices to consider  cubes $Q_j \in \mathcal Q$.
	
By \eqref{Bbox}, \eqref{conf-dcp}, and H\"older's inequality, we get
\begin{align}\label{ineq1}
\| e^{it\phi} f \|_{L^6(Q_j)}
&\le C_\varepsilon (\log R)^4 R^{\varepsilon} \mu^{\frac 13}\big( \sum_{\substack{T}} \|  \chi_{\cup T'} e^{it\phi} f_T 
\|_{L^6(\omega_{Q_j})}^6
\big)^{1/6}  +O(R^{-N})\|f\|_2.
\end{align}
Since $T'= \mathcal A_0^{-1} (Q')$ for some cube $Q'$ of side length $R^{1/4}$,
by \eqref{frho}, we have
\begin{align*}
\sum_{Q_j } \| \chi_{\cup T'}  e^{it\phi} f_T \|_{L^6(\omega_{ Q_j})}^6
\lesssim  
R^{-\frac 12} \| e^{it \widetilde \phi} \widetilde f_{T}
\|_{L^6(\cup Q')}^6
\end{align*}
where $\mathcal F(\widetilde f_{T})$ is given by \eqref{tildef} for $f_T$ in place of $f$.
Then, $\mathcal F(\widetilde f_T)$ is supported on $B^2(0,1)$ and $\| \widetilde f_T\|_2 = \| f_T\|_2$.
 Combining this with \eqref{ineq1}, we have
\begin{align}\label{ineq111}
\| e^{it\phi} f \|_{L^6(\cup_j Q_j)}
\le C_\varepsilon  R^{-\frac 1{12}+ 2\varepsilon} \mu^{\frac 13}
\Big(\sum_{T }\| e^{it \widetilde \phi} \widetilde f_{T}
\|_{L^6(\cup Q')}^6\Big)^{\frac 16}
+O(R^{-N})\|f\|_2.
\end{align}
By the choice of \eqref{t'} combined with \eqref{frho} for $f=f_T$,
we observe that $\| e^{it \widetilde \phi} \widetilde f_{T} \|_{L^6(Q')}$
is essentially constant for each $Q'$.
Since $\phi \in \mathcal P(\epsilon_0)$, we have $\widetilde \phi \in \mathcal P(\epsilon_0)$
by Lemma \ref{reducephi} provided that $R^{-1/2} \le \rho_1$.
Thus, by applying the induction hypothesis to the right-hand side of \eqref{ineq111} with $R^{1/2}$ and $M'$ instead of $R$ and $M$, respectively,  we have
\begin{align}\label{key}
\| e^{it\phi} f\|_{L^6(\cup_jQ_j)}&
\le C_\varepsilon   
	R^{-\frac 1{6}+2\epsilon+\frac12C\delta'} \mu^{\frac 13}(M')^{-\frac 13}
\big(\sum_{T}  \|f_T \|_2^6\big)^{\frac 16}.
\end{align}
Since $\|f_T\|_2$ is essentially constant, it is easy to see that
$
(\sum_{T}  \|f_T \|_2^6)^{1/6} \lesssim (\# T)^{-1/3} \|f\|_2.
$
We claim that
\begin{align}\label{MM'}
M\mu \lesssim (\log R)^4 M'\#T.
\end{align}
Once we have \eqref{MM'}, \eqref{key} is, in turn, bounded by
\begin{align*}
\| e^{it\phi} f\|_{L^6(\cup_jQ_j)}
\le C_\varepsilon R^{-\frac 16+3\varepsilon+\frac 12C\delta'}M^{-\frac 13} \|f\|_2.
\end{align*}
Taking $\varepsilon =\epsilon/C$ for sufficiently large $C>0$ gives the desired bound \eqref{unionQ}.

Now, we prove \eqref{MM'}.
Let $S^{(t_0)}$ be the strip given by 
\[
S^{(t_0)}=\mathbb R^2 \times [t_0,t_0+R^{1/2}]
\]
for some $t_0 \in [0,R] \cap R^{1/2}\mathbb Z$.
It suffices to consider cubes $Q_j$
arranged along the strip $S^{(t_0)}$.
By the choice of $Q_j$, $T$, and $T'$, we note that
\begin{align}\label{QQ}
(\log R)^{-4} M\mu 
\lesssim
\# \{ (Q_j, T) : Q_j \subset T' \cap S^{(t_0)} \,\,\text{for some}\,\, T' \subset T\}.
\end{align}
Note that the angle between the long axis of $T'$ and
the vectors contained in the $x$-plane $\mathbb R^2$ is away from $0$.
Hence, each tube $T'$ contains at most a finite number of cubes $Q_j$ in $S^{(t_0)}$.
Similarly, since $T' \subset \mathcal A_0^{-1}(S')$ for some $S'$, the number of $S'$ such that $S^{(t_0)} \cap \mathcal A_0^{-1} (S') \neq \emptyset$ is at most $C$. Hence, the number of $Q_j$ is smaller than $C$ times that of $T'$ such that $T' \subset \mathcal A_0^{-1}(S')$. Therefore, $\#Q_j \le C M'$. 
Thus, the right-hand side of \eqref{QQ} is bounded by $M'\#T$, which gives the desired bound \eqref{MM'}.
\end{proof}

\subsection{Bilinear refined Strichartz estimates}

Now, we establish the bilinear refined Strichartz estimates using Proposition \ref{linear}.

\begin{prop}\label{bilinearM}
Let  $\phi \in \mathcal P(\epsilon_0)$, $R^{-1/2}\le r \le 1$,
and $K$ be a sufficiently large constant. 
Suppose that the Fourier supports of $f$, $g$ are contained in $B^2(\xi_0,r)$
and separated by $K^{-1}r$,
and $f$, $g$ are concentrated on the wave packets from
$T_Z(R^{-1/2+\delta'})$ for $Z=Z(P)$ with 
$\delta' \le 100\delta$.
Let $Q_1,\dots,Q_M$ be cubes of side length $R^{1/2}$ contained in $B_R$ and
\begin{equation}\label{constM}
\big \| |e^{it \phi} f
e^{it \phi} g|^{\frac 12} \big\|_{L^6(Q_j)} \quad \text{is essentially constant}
\end{equation}
for each $Q_j$.
Then, for any $\epsilon>0$,
there exist $C_\epsilon, C\ge 1$ such that
\begin{align}\label{BM}
\big\| |e^{it \phi} f
e^{it \phi} g|^{\frac 12} \big\|_{L^6(\cup_{j=1}^MQ_j)}
\le C_\epsilon  r^{-\frac 16}M^{-\frac 16}R^{-\frac 16+\epsilon+C\delta'}
\|f\|_2^{\frac12} \| g \|_2^{\frac12}.
\end{align}
\end{prop}

Before proving Proposition \ref{bilinearM},
we first consider the special case $r=1$ of Proposition \ref{bilinearM} as follows, which is a direct consequence of Proposition \ref{linear}.

\begin{prop}\label{bilinear}
Let $\phi \in \mathcal P(\epsilon_0)$.
Suppose that the Fourier supports of $f,g$ are contained in $B^2(0,1)$
and separated by $K^{-1}$, and
$f$ and $g$ are concentrated on 
the wave packets from $T_Z(R^{-1/2+\delta'})$
for $\delta'\le 100\delta$.
Suppose that $Q_1,\dots,Q_M$ are cubes of side length $R^{1/2}$, and
$\| e^{it \phi} f \|_{L^6(Q_j)}$ and $\| e^{it \phi} g \|_{L^6(Q_j)}$ are essentially constant
for each $Q_j$. Then, for any $\epsilon>0$,
there exist $C_\epsilon, C \ge 1$ such that
\begin{equation}\label{B}
\big\| |e^{it \phi} f e^{it \phi} g |^{\frac 12} \big\|_{L^6(\cup_jQ_j)} \le C_\epsilon R^{-\frac 16+\epsilon+C\delta'}M^{-\frac 16} \|f\|_2^{\frac 12} \|g\|_2^{\frac 12}.
\end{equation}
\end{prop}

\begin{proof} 
Let us consider the wave packet decompositions
$f=\sum_T f_T$ and $g=\sum_{\widetilde T}g_{\widetilde T}$ (see \eqref{FT}).
By repeating the pigeonholing arguments as in the proof of Proposition \ref{linear},
we can choose $T_f$, $T_f'$, $M_f'$, $\mu_f$ and $T_g$ $T_g'$, $M_g'$, $\mu_g$,
respectively, in place of $T$, $T'$, $M$, $\mu$.
Then, by pigeonholing, there are $C(\log R)^{-8}M$ cubes $Q_j$  
such that \eqref{key} holds for $f$ and $g$ simultaneously.
We claim that
\begin{align}\label{nhh}
M \mu_f \mu_g &\lesssim M_f' M_g' \# T_f
\# T_g.
\end{align}
First, by assuming \eqref{nhh}, we prove \eqref{B}.
By H\"older's inequality,
\[
\big\| |e^{it \phi} f e^{it \phi} g |^{\frac 12} \big\|
_{L^6(\cup_jQ_j)}
\le  \| e^{it \phi} f \|_{L^6(\cup_jQ_j)}^{\frac 12}
\| e^{it \phi} g\|_{L^6(\cup_jQ_j)}^{\frac 12}.
\]
We apply \eqref{key} to the right-hand side and get
\begin{align*}
\big\| |e^{it \phi} f e^{it \phi} g |^{\frac 12} \big\|
_{L^6(\cup_jQ_j)}
&\le
C_\epsilon R^{-\frac 16+\epsilon+\frac 12C\delta'} 
\big( \mu_f \mu_g \big)^{\frac 16}
(\#T_f \#{T_g} M_f' M_g')^{-\frac 16}  \|f\|_2^{\frac 12}
\|g \|_2^{\frac 12}.
\end{align*}
Substituting \eqref{nhh}, we obtain the desired bound \eqref{B}.

Now, we prove \eqref{nhh}. For each fixed $T_f$ and $T_g$,
we first consider the intersection of $T_f' \subset T_f$ and $T_g'\subset T_g$.
Since the angle between $T_f'$ and $T_g'$ is about $1/K$,
the intersection of $T_f'$ and $T_g'$ is essentially a cube of side length $C_KR^{1/2}$.
Therefore, if cubes $Q_j$ are contained in $(\cup T_f') \cap (\cup T_g')$,
then the number of such $Q_j$ is at most $CM_f' M_g'$.

Recall that $\mu_f$ is the number of $T_f$ containing $Q_j$ such that $Q_j \subset T_f'$
and $\mu_g$ similarly.
Considering all pairs $\{  T_f, T_g, Q_j\}$ such that $Q_j \subset T_f' \cap T_g'$ for some $T_f' \subset T_f$
and $ T_g' \subset  T_g$, we obtain \eqref{nhh}.
\end{proof}

Now, we prove Proposition \ref{bilinearM}. If $r\sim1$, Proposition \ref{bilinearM} is a direct consequence of
Proposition \ref{bilinear}. Hence, we are only concerned with $R^{-1/2} \le r \ll1$. In this case, Proposition \ref{bilinearM} follows from Proposition \ref{bilinear} via rescaling in Lemma \ref{lem:scale}.

\begin{proof}[Proof of Proposition \ref{bilinearM}]
Let $\phi \in \mathcal P(\epsilon_0)$.
By \eqref{map},
$\mathcal A_{\xi_0}^r(B_R)$ is a tube of dimensions $Cr{ R} \times Cr{R} \times r^2{R}$
for a constant $C\ge1$.
Let us set
\[
P_j=\mathcal A_{\xi_0}^r(Q_j).
\]
Since $Q_j$ is a cube of side length $R^{1/2}$, $P_j$ is a tube of dimensions $CrR^{\frac 12} \times CrR^{\frac 12} \times r^2R^{\frac12}$ contained in a larger tube $\mathcal A_{\xi_0}^r(B_R)$.

We define $\widetilde f$ and $\widetilde g$ by \eqref{tildef} for $f$ and $g$, respectively, with $r=\rho$.
Further, let $\widetilde \phi=\phi_{\xi_0}^r$.
By Lemma \ref{reducephi}, we see that $\widetilde \phi \in \mathcal P(\epsilon_0)$ by taking sufficiently small $r\le \rho_1$.
Moreover, by  Lemma \ref{lem:scale}, we have
\[
\big\| |e^{it \phi} f
e^{it \phi} g|^{\frac12} \big\|_{L^6(\cup_{j=1}^MQ_j)}
\lesssim r^{\frac 13} 
\big\| |e^{it \widetilde \phi} \widetilde f \,\,
e^{it \widetilde \phi} \widetilde g|^{\frac 12} \big\|_{L^6(\cup_j P_j)}
\]
such that $\|\widetilde f\|_2 = \|f\|_2$, $\|\widetilde g \|_2=\|g\|_2$
and the Fourier supports of $\widetilde f$ and $\widetilde g$
are contained in $B^2(0,1)$ and separated by $K^{-1}$.
Moreover, by Lemma \ref{tan_phirho}, $\widetilde f$ and $\widetilde g$ are $r^{-1}R^{-\frac12+\delta'}$-concentrated in the wave packets from $T_Z(r^{-1}R^{-\frac12+\delta'})$ associated with $\widetilde \phi$ for $Z=Z(\widetilde P)$ for some polynomial $\widetilde P=P \circ (\mathcal A_{\xi_0}^r)^{-1}$.
Thus, to prove \eqref{BM}, it suffices to show that
\begin{equation}\label{A3}
\big\| |e^{it \widetilde \phi} \widetilde f \, e^{it \widetilde \phi} \widetilde g\,|^{\frac 12} \big\|_{L^6(\cup_j P_j)}
\le C_\epsilon  r^{-\frac12} M^{-\frac 16}R^{-\frac 16+\epsilon+C\delta'}
\|\widetilde f\|_2^{\frac 12} \| \widetilde g \|_2^{\frac 12}.
\end{equation}

Now, we make couple of reductions by pigeonholing.
For this purpose, let us set 
\begin{equation}
\begin{aligned}\label{r12}
\begin{cases}
r_1 = r^2R, \\
r_2 = rR^{1/2}.
\end{cases}
\end{aligned}
\end{equation} 
Recall that $\mathcal A_{\xi_0}^r(B_R)$ is a tube of dimensions $Cr^{-1}r_1\times Cr^{-1} r_1 \times r_1$.
Now, we decompose $\mathcal A_{\xi_0}^r(B_R)$ into cubes $Q_{r_1}$ of side length $r_1$. 
For each fixed cube $Q_{r_1}$,
we consider cubes $Q_{r_2}'$ of side length $r_2$ contained in $Q_{r_1}$
such that $\{ Q_{r_2}'\}$ covers $(\cup_jP_j) \cap Q_{r_1}$.
By dyadic pigeonholing on the size of $\|e^{it \widetilde \phi} \widetilde f \|_{L^2( Q_{r_2}'  )}$ and $\|e^{it \widetilde \phi} \widetilde g \|_{L^2( Q_{r_2}'  )}$, we may sort $Q_{r_2}'$,  which satisfy
\begin{equation}\label{sort3}
\|e^{it \widetilde \phi} \widetilde f \|_{L^2( Q_{r_2}'  )}, \quad~ 
\|e^{it \widetilde \phi} \widetilde g \|_{L^2( Q_{r_2}'  )} 
\quad \text{are essentially constant.}
\end{equation}
Let $\mathcal Q_{r_2}'$ be the collection of such cubes $Q_{r_2}'$.
We also choose a subcollection of $\{ Q_{r_1}\}$ that satisfy
\begin{align}\label{sort1}
& \big\| |e^{it \widetilde \phi} \widetilde f \,\, e^{it \widetilde \phi} \widetilde g|^{\frac 12} \big\|_{L^6(\cup_{Q_{r_2}'\in \mathcal Q_{r_2}'} Q_{r_2}' )}
\quad~ \text{is essentially constant}
\end{align}
for $Q_{r_2}' \in \mathcal Q_{r_2}'$ such that $Q_{r_2}' \subset Q_{r_1}$, and also satisfy
\begin{align}
&\|e^{it\widetilde \phi} \widetilde f \|_{L^2(R^{2a} Q_{r_1})},  
\quad~
\|e^{it\widetilde \phi}\, \widetilde g \|_{L^2(R^{2a} Q_{r_1})} 
\quad~ \text{are essentially constant}. \label{sort2}
\end{align}
Here, $R^{2a} Q_{r_1}$ is a cube that has the same center as $Q_{r_1}$ 
and the side length is $R^{2a}r_1$ for a sufficiently small constant $a=a(\epsilon)$, which will be determined later.
We denote such a collection of $Q_{r_1}$ by $\mathcal Q_{r_1}$.
We also denote by $\mathcal P=\mathcal P(r_1,r_2)$ the collection of tubes $P_j$
that intersect with $ Q_{r_2}' \subset Q_{r_1}$ for $Q_{r_2}'\in \mathcal Q_{r_2}'$  and $ Q_{r_1} \in \mathcal Q_{r_1}$.
By dyadic pigeonholing, we have
\begin{equation}\label{P1}
\# \mathcal P 	\gtrsim (\log R)^{-5} M.
\end{equation}

We put 
\begin{align*}
M_1 &= \#\{ Q_{r_1} \in \mathcal Q_{r_1}\},\quad
M_2  =\#\{ Q_{r_2}' \in \mathcal Q_{r_2}'\}.
\end{align*}
Since the number of $P_j \in \mathcal P$ contained in a cube $Q_{r_2}'$ is at
most $r^{-1}$, it is easy to see that
\begin{align}\label{MM12}
(\log R)^{-5} M \lesssim r^{-1} M_1M_2.
\end{align}
Thus, from \eqref{constM} and \eqref{P1}, we get
\begin{align}\label{sort13}
\big\| |e^{it \widetilde \phi} \widetilde f \,\, e^{it \widetilde \phi} \, \widetilde g |^{\frac 12} 
\big\|_{L^6(\cup_j P_j)}^6
&\lesssim (\log R)^5 \big\| |e^{it \widetilde \phi} \widetilde f \,\, e^{it \widetilde \phi} \, \widetilde g |^{\frac12}
\big\|_{L^6(\cup_{P_j \in \mathcal P}P_j)}^6.
\end{align}
By \eqref{sort1}, for any $Q_{r_1} \in \mathcal Q_{r_1}$ and $\cup Q_{r_2}' \subset Q_{r_1}$, we have
\begin{align}\label{unionPM}
\big\| |e^{it \widetilde \phi} \widetilde f\,\,  e^{it \widetilde \phi} \,\widetilde g |^{\frac 12} \big\|_{L^6(\cup_{P_j \in \mathcal P} P_j)}
&\lesssim 
M_{1}^{\frac 16} 
\big\| |e^{it \widetilde \phi} \widetilde f \,\, e^{it\widetilde \phi} \, \widetilde g |^{\frac12}  \big\|_{L^6(
\cup Q_{r_2}')}.
\end{align}

The right-hand side of \eqref{unionPM} can be estimated by applying the induction hypothesis.
By Lemma \ref{wave packet}, we have $\widetilde f=\sum_{T} \widetilde f_{T}$ and $\widetilde g=\sum_{T'} \widetilde g_{T'}$.
Here, the Fourier transforms of $ \widetilde f_{T}$ and $\widetilde g_{T'}$
are supported on balls of radius $r_1^{-1/2}$,
and $e^{it \widetilde \phi} \widetilde f_{T}$ and $e^{it \widetilde \phi} \widetilde g_T$ restricted to $B^3(0,r_1)$ are essentially supported on tubes $T$, $T'$ of dimensions $\sim r_1^{1/2}\times r_1^{1/2} \times r_1$.

Note that it suffices to consider the wave packets that intersect with $Q_{r_1}$.
Indeed, let us fix $t_0$ such that $(x_0,t_0) \in Q_{r_1}$ for some $x_0$
and denote the slice of $Q_{r_1}$ by
\[
Q_{r_1}^{(t_0)} = Q_{r_1} \cap \{ t=t_0 \}.
\]
Since $\cup Q_{r_2}' \subset Q_{r_1}$, to estimate the right-hand side of \eqref{unionPM}, 
it suffices to consider the wave packets $T, T'$ intersecting with $R^a Q_{r_1}^{(t_0)}$, which is an $R^a$-neighborhood of $Q_{r_1}^{(t_0)}$. Therefore,
\begin{align*}
&\big\|
|e^{it \widetilde \phi} \widetilde f \,\, e^{it \widetilde \phi} \widetilde g |^{\frac 12}  \big\|_{L^6(\cup Q_{r_2}')}
\lesssim
\big\| | \sum_{\substack{T \cap R^aQ_{r_1}^{(t_0)} \neq \emptyset, \\
T' \cap R^a  Q_{r_1}^{(t_0)} \neq \emptyset}}
e^{it \widetilde \phi} \widetilde f_{T} e^{it \widetilde \phi} \widetilde g_{T'} |^{\frac12}  \big\|_{L^6(\cup Q_{r_2}')}
+\mathcal E_N \nonumber
\end{align*}
where $\mathcal E_N =O(R^{-N}) \|f\|_2^{1/2} \|g\|_2^{1/2}$ for sufficiently large $N$.
Under the assumption \eqref{sort3},
we apply \eqref{B} for $R=r_1$ to the above inequality and using Plancherel's theorem, we obtain
\begin{align*}
&\big\| |e^{it \widetilde \phi} \widetilde f\,\, e^{it \widetilde \phi} \widetilde g |^{\frac12}  \big\|_{L^6(\cup Q_{r_2}')} 
\le
C_{\varepsilon}r_1^{-\frac16+\varepsilon+C\delta' }M_2^{-\frac16}
\big\| \sum_{T}
e^{it_0 \widetilde \phi}  \widetilde f_{T} \big\|_{L_x^2}^{\frac12}
\big \|\sum_{T'}
e^{it_0 \widetilde \phi} \widetilde g_{T'}\big\|_{L_x^2}^{\frac12}+\mathcal E_N
\end{align*}
where the summation is taken over $T$ and $T'$ satisfying $ T \cap R^a Q_{r_1}^{(t_0)} \neq \emptyset$ and
$T' \cap R^a Q_{r_1}^{(t_0)} \neq \emptyset$.

Note that the length $Cr_1^{1/2}$ of the short axis  of a tube $T$ is smaller than the side length $R^ar_1$ of the cube $R^aQ_{r_1}$.
Hence, the intersection of the tube $T$ and $Q_{r_1}^{(t_0)}$ is contained in $R^{2a}Q_{r_1} \cap \{ t=t_0\}$
for sufficiently large $R>0$.
Thus, we have
\begin{equation}
\begin{aligned}\label{L2est0}
	\big\| |e^{it \widetilde \phi} &\widetilde f \,\, e^{it \widetilde \phi} \widetilde g |^{\frac 12} \big \|_{L^6(\cup Q_{r_2}')}  \\
	 &\quad\le \!C_\varepsilon r_1^{-\frac16+\varepsilon+C\delta'}  M_2^{-\frac16}
	 \| e^{it_0 \widetilde \phi } \widetilde f\|_{L_{x}^2(R^{2a}Q_{r_1}^{(t_0)})}^{\frac12}
\|  e^{it_0 \widetilde \phi} \widetilde g\|_{L_{x}^2( R^{2a} Q_{r_1}^{(t_0)})}^{\frac12}
\!+\!\mathcal E_N.
\end{aligned}
\end{equation}

By taking the average with respect to $t$, we have 
\begin{align}\label{L2est1}
\| e^{it_0 \widetilde \phi } \widetilde f \|_{L_{x}^2(R^{2a} Q_{r_1}^{(t_0)})}
\le
\big( r^2R^{1+2a} \big)^{-\frac12}
\| e^{it \widetilde \phi}  \widetilde f \|_{L^2(R^{2a} Q_{r_1})}.
\end{align}
Since $\cup R^{2a}Q_{r_1} 
\subset \mathcal A_{\xi_0}^r(B_{R^{1+3a}})$,
by \eqref{sort2} and Lemma \ref{lem:scale}, we get
\begin{align}\label{L2est2}
\| e^{it\widetilde \phi} \widetilde f \|_{L^2(R^{2a}Q_{r_1})}
&\lesssim 
M_1^{-\frac12} \| e^{it\widetilde \phi} \widetilde f \|_{L^2(\mathrm A_{\xi_0}^{r}(B_{R^{1+3a}}))}
\lesssim  
M_1^{-\frac 12} r \|e^{it\phi} f\|_{L^2(B_{R^{1+3a}})}.
\end{align}
By combining \eqref{L2est1} and \eqref{L2est2} with the trivial estimate $ \|e^{it\phi} f\|_{L^2(B_{R^{1+3a}})}\lesssim R^{(1+3a)/2}\|f\|_2$ via Plancherel's theorem, we have
\[
\big\| e^{it_0 \widetilde \phi } \widetilde f \big\|_{L_{x,t}^2(R^{2a} Q_{r_1}^{(t_0)})}
\lesssim M_1^{-\frac12} R^{\frac {a}2}\|f\|_2.
\]
Since $r_1=r^2R$ (\eqref{r12}), disregarding the error term $\mathcal E_N$, \eqref{L2est0}  is bounded by
\begin{align*}
&\big\| |e^{it \widetilde \phi} \widetilde f \,\, e^{it \widetilde \phi} \widetilde g |^{\frac 12} \big \|_{L^6(\cup Q_{r_2}')} 
\le
C_{\varepsilon} 
(r^2R)^{-\frac16+\varepsilon+C\delta'} M_2^{-\frac16}
M_1^{-\frac 12} R^{\frac a2} \|f\|_2^{\frac 12} \|g\|_2^{\frac 12}.
\end{align*}
By combining this with \eqref{sort13} and \eqref{unionPM}, 
and taking sufficiently small $a=a(\varepsilon)$ for $R^{-1/2} \le r$, we obtain
\begin{align*}
\big\| |e^{it\widetilde \phi} \widetilde f e^{it\widetilde \phi} \widetilde g |^{\frac 12} \big\|_{L^6(\cup_j P_j)}
&\le C_{\varepsilon} R^{-\frac16+3\varepsilon+C\delta'}
r^{-\frac 13} 
M_{1}^{-\frac13} M_2^{-\frac16}
\| f\|_2^{\frac12} \|g\|_2^{\frac12}.
\end{align*}
Therefore, using \eqref{MM12},  the above is bounded by
$ C_{\varepsilon} R^{-\frac16+4\varepsilon+C\delta'}  r^{-\frac12} M^{-\frac16}	
M_{1}^{-\frac16} \| f\|_2^{\frac12} \|g\|_2^{\frac12}$.
Since $M_{1} \ge 1$, by taking $\varepsilon \le \epsilon/C_1$ for sufficiently large $C_1\ge1$, we obtain \eqref{A3}. This completes the proof of \eqref{BM}.
\end{proof}

We conclude this section by proving Theorem \ref{8.1}.
\subsection{Proof of Theorem \ref{8.1}.}

Let $\epsilon>0$.
Assume that $\widehat f$, $\widehat g $ are supported on $B^2(\xi_0,r)$ for some $0<r< 1$ and $\xi_0 \in \interior \mathbb A_1$. From \eqref{one}, we may assume that $r\ge R^{-1/2}$.
For a dyadic number $ \lambda$, we define a set $X_\lambda $ by
\[
X_\lambda = \{ x \in B^2(0,R) : 
\sup_{0<t<R}  |U_\alpha f(x) U_\alpha g(x)| \sim \lambda \}.
\]
It suffices to consider  $R^{-N} \le \lambda \le R^{c}$.

We cover $N_{R^{1/2+\delta'}}(Z)$ with cubes $Q_j$, $j=1,\dots,M$, of side length $R^{1/2}$.
Since $\mathcal F(U_\alpha f)$ is supported on a ball of radius $r$, we consider smaller cubes $Q'=Q'(Q_j)$
of side length $r^{-1} \le R^{1/2}$ such that $Q' \subset Q_j$.
We define $S_\lambda$ by the union of $Q'$ contained in $B_R$ such that
\[
\sup_{(x,t) \in Q'} |U_\alpha f(x) U_\alpha g(x) |\gtrsim \lambda
\]
and all projections of $Q'$ onto the $x$-plane $\mathbb R^2$ are boundedly overlapping. 
Clearly, we have
\begin{equation}\label{ulambda} 
|X_\lambda| \lesssim r |S_\lambda|.
\end{equation}
For each $Q_j$, the number of $Q'$ contained in $Q_j$ is at most $C(R^{1/2}/r^{-1})^2=CRr^2$ since the projections of $Q'$ onto $x$-space are finitely overlapping. Hence, by pigeonholing, we also have
\begin{align}\label{xlambda}
|S_\lambda| 
\lesssim (\log R) M Rr^{-1}. 
\end{align}

Now, we prove Theorem \ref{8.1}. By construction, it suffices to show that
\begin{align}\label{eq0101}
\lambda^{\frac12} |X_\lambda|^{\frac13} \le C R^{c\delta'} \|f\|_2^{\frac 12} \|g\|_2^{\frac 12}.
\end{align}
After finite decomposition, we may assume that
$\widehat f$ and $\widehat g $ are supported on a ball $B^2 (\eta_0,\epsilon_1 r)$ for some 
small $\epsilon_1$ and $\eta_0 \in \interior \mathbb A_1$ since $f$ and $g$ have Fourier supports separated by $K^{-1}r$ 
such that $K^{-1} \ll \epsilon_1$.
By Lemmas \ref{reducephi} and \ref{tan_phirho} with $\rho=\epsilon_1 >0$ and $\xi_0 = \eta_0$, 
we have $f_{\epsilon_1}=\widetilde f$, $g_{\epsilon_1}=\widetilde g$ (\eqref{tildef}), and $\phi_{\epsilon_1}=\phi_{\eta_0}^{\epsilon_1}$ for $\phi(\xi)=|\xi|^\alpha$
such that $\phi_{\epsilon_1} \in\mathcal P(\epsilon_0)$ and
 $f_{\epsilon_1}$ and $g_{\epsilon_1}$ have Fourier supports on a ball $B^2(0,r)$ separated by $\epsilon_1^{-1}K^{-1}r$ and concentrated on the wave packets from $T_Z(\epsilon_1^{-1}R^{-1/2+\delta'})$.

By pigeonholing on the size of $\| |e^{it\phi_{\epsilon_1}} f_{\epsilon_1}  e^{it\phi_{\epsilon_1}}g_{\epsilon_1} |^{1/2} \|_{L^6(Q_j)}$ , we sort the fraction $\sim 1/(\log R)$ of cubes $Q_j$ such that 
\begin{equation}\label{qh1h2}
\big\| |e^{it\phi_{\epsilon_1}} f_{\epsilon_1}  e^{it\phi_{\epsilon_1}}g_{\epsilon_1} |^{\frac 12} \big\|_{L^6( Q_j)} 
\quad \text{is essentially constant.}
\end{equation}
We denote by $\mathcal Q$ the set of such cubes $Q_j$.
Thus, to prove \eqref{eq0101}, it suffices to show that
\begin{equation}\label{lambdax}
\lambda^3 |S_\lambda|
\lesssim  R^{c_1\delta'} \big\| |e^{it \phi_{\epsilon_1}}f_{\epsilon_1} 
e^{it \phi_{\epsilon_1}}g_{\epsilon_1}|^{\frac12} \big\|_{L^6( \cup_{Q_j \in \mathcal Q}  Q_j)}^6
\end{equation}
for some constant $c_1$. Indeed, by \eqref{ulambda}, we have
$
|X_\lambda|^{\frac13}
\lesssim 
r^{\frac 13} |S_\lambda|^{\frac16}  |S_\lambda|^{\frac16}
$
and $r^{\frac 13}|S_\lambda|^{\frac 16} \le (\log R)^{\frac 16}(rMR)^{\frac 16} $ by \eqref{xlambda}. Hence, we have
\[
\lambda^{\frac 12}|X_\lambda|^{\frac 13} \le
(\log R)^{\frac 16}(rMR)^{\frac16}  (\lambda^3 |S_\lambda|)^{\frac16} .
\]
Combining this with \eqref{lambdax}, we obtain
\begin{align*}\label{eq111}
\lambda^{\frac12} |X_\lambda|^{\frac13} 
\lesssim
r^{\frac16} M^{\frac16} R^{\frac16+c\delta'} 
\big\| |e^{it \phi_{\epsilon_1}}f_{\epsilon_1} 	e^{it \phi_{\epsilon_1}}g_{\epsilon_1} |^{\frac12} \big\|_{L^6(\cup_{Q_j \in \mathcal Q} Q_j)}.
\end{align*}
By applying Proposition \ref{bilinearM} to the right-hand side, we get the desired bound \eqref{eq0101} and hence \eqref{bimax}.

It remains to show \eqref{lambdax}. Let $Q' \subset Q_j$ and let $\psi_{Q'} \in \mathcal C_0^\infty(\mathbb R^3)$ such that $\mathcal F(\psi_{Q'})$
is supported on $B^3(0,2r)$ and $\psi_{Q'}$ decreases rapidly outside $Q'$. 
Since $\widehat f$, $\widehat g$ are supported on $B^2(\xi_0,r)$, we see that
\[
\mathcal F(U_\alpha f U_\alpha g)(\xi,\xi_3)
=\mathcal F( U_\alpha f U_\alpha g)(\xi,\xi_3) \mathcal F(\psi_{Q'})(\xi,\xi_3).
\]
Since $\psi_{Q'}$ decreases rapidly outside a ball of radius $r^{-1}$,
$\int U_\alpha f U_\alpha g ((x,t)-(y,s))\psi_{Q'}(y,s)\,dyds$ is negligible when $|(y,s)| \ge R^{a}r^{-1}$ for sufficiently small $a$.
Hence, we have
\begin{equation*} 
\sup_{(x,t) \in Q'}  |U_\alpha f(x,t) U_\alpha g(x,t)|
\lesssim 
r^3 \int_{R^{a}Q'} |U_\alpha f(x,t) U_\alpha g(x,t) | \,dxdt
+E_N
\end{equation*}
where $E_N=O(R^{-N}\|f\|_2\|g\|_2)$ for sufficiently large $N$.
By applying Lemma \ref{lem:scale} with $\rho=\epsilon_1$ and $\xi_0=\eta_0$, we get
\begin{align}\label{keyreduction}
|Q'|\sup_{(x,t) \in Q'} 
|U_\alpha f(x,t) U_\alpha g(x,t)|
\lesssim 
\int_{\widetilde A_{Q'}} 
| e^{it \phi_{\epsilon_1}}f_{\epsilon_1} (x) e^{it \phi_{\epsilon_1}}g_{\epsilon_1} (x) | \,dxdt
+E_N
\end{align}
where $\widetilde A_{Q'}=\mathcal A_{\eta_0}^{\epsilon_1}(R^{a} Q')$.
By applying H\"older's inequality, we have
\begin{align*} 
|Q'|^{\frac 13}\sup_{(x,t) \in Q'} &|U_\alpha f(x,t) U_\alpha g(x,t)|  
\lesssim 
R^{ca}
\Big( \int_{ \widetilde A_{Q'}} | e^{it \phi_{\epsilon_1}}f_{\epsilon_1}(x) e^{it \phi_{\epsilon_1}}g_{\epsilon_1}(x) |^3 \,dxdt \Big)^{\frac 13}+E_N
\end{align*}
for some constant $c>1$.
If $a$ is sufficiently small, then $\widetilde A_{Q'}$ is contained in a cube of side length $R^{1/2}$;
hence, we may assume that $\widetilde A_{Q'} \subset Q_j$.
Taking the third power and summing over all $Q'$ such that $\widetilde A_{Q'} \subset \cup_j Q_j$ and using \eqref{qh1h2}, we have \eqref{lambdax} with a minor error term $E_N$. This completes the proof.
\qed

\begin{rmk}\label{disp}
Without difficulty, we can see that the result of Theorem \ref{point3} 
extends to the solution of the general dispersive equation.
Let $\alpha>1$ and $\Phi$ be a smooth function such that $ |\nabla \Phi(\xi)| \ge C^{-1} |\xi|^{\alpha-1}$ and $|\partial_\xi^\gamma \Phi(\xi)| \le C |\xi|^{\alpha-|\gamma|}$ for any $\gamma$ for some constant $C>0$.
The reduction to the bilinear estimate in Theorem \ref{8.1} (see Section \ref{sec2}) via polynomial partitioning is nearly identical.
Since $\Phi$ is positive definite, by rescaling and Lemma \ref{reducephi}, 
 we only need to consider 
$e^{it\phi}$ for $\phi \in \mathcal P(\epsilon_0)$ (see \eqref{keyreduction}).
Then, by following the proof of Theorem \ref{8.1}, we can obtain \eqref{bimax}
for $e^{it\Phi}$ in place of $U_\alpha f$.
\end{rmk}

\begin{rmk}
A modification of the proof of Theorem \ref{point3} combined with the argument in \cite{DZ} provides \eqref{convs} for $d\ge3$ whenever $s> d/(2d+2)$.
\end{rmk}

\bibliographystyle{plain}

\end{document}